\documentclass[10pt,a4paper]{article}
\usepackage{amsthm,amsmath,amssymb}
\usepackage{graphicx}
\usepackage{siunitx}
\usepackage{xcolor}
\usepackage{subcaption}
\usepackage{hyperref}
\usepackage{tikz,graphics,fullpage,float}
\usepackage{mathtools}
\usepackage{aligned-overset}
\usepackage{enumitem}
\usepackage{booktabs}
\usepackage{url}
\usepackage[T1]{fontenc}
\usepackage[linesnumbered,ruled]{algorithm2e}

\usepackage[most]{tcolorbox}

\tcbset{
	colback=gray!5!white,  
	colframe=blue!15!white, 
	fonttitle=\bfseries,   
	coltitle=black,        
	boxrule=0.8pt,         
	arc=4pt,               
	left=6pt, right=6pt, top=6pt, bottom=6pt,
	enhanced, breakable  
}

\usetikzlibrary{calc, positioning, fit, shapes.misc}
\usetikzlibrary{topaths,calc}

\usepackage{booktabs}

\usepackage{graphicx}

\let\emph\relax 
\DeclareTextFontCommand{\emph}{\color{blue}\bfseries\em}

\usepackage[noabbrev,capitalize]{cleveref}
\crefname{equation}{}{}

\newtheorem{theorem}{Theorem}[section]
\newtheorem{proposition}[theorem]{Proposition}
\newtheorem{lemma}[theorem]{Lemma}
\newtheorem{corollary}[theorem]{Corollary}

\theoremstyle{definition}
\newtheorem{definition}[theorem]{Definition}
\newtheorem{problem}[theorem]{Problem}

\newtheorem{claim}[theorem]{Claim}
\newtheorem{remark}[theorem]{Remark}

\newcommand{\FF}{\mathbb{F}}

\newcommand{\PP}{\mathbb{P}}

\newcommand{\cB}{\mathcal{B}}
\newcommand{\cC}{\mathcal{C}}

\newcommand{\cF}{\mathcal{F}}

\newcommand{\cH}{\mathcal{H}}

\newcommand{\cO}{\mathcal{O}}

\newcommand{\cT}{\mathcal{T}}

\newcommand{\floor}[1]{\left\lfloor #1 \right\rfloor}

\newcommand{\wt}{\widetilde}

\makeatletter
\newcommand*\bigcdot{\mathpalette\bigcdot@{.5}}
\newcommand*\bigcdot@[2]{\mathbin{\vcenter{\hbox{\scalebox{#2}{$\m@th#1\bullet$}}}}}
\makeatother

\DeclareMathOperator{\ex}{ex}

\DeclareMathOperator{\ar}{ar}

\DeclareMathOperator{\STS}{STS}
\DeclareMathOperator{\PSTS}{PSTS}
\DeclareMathOperator{\MPSTS}{MPSTS}
\DeclareMathOperator{\PG}{PG}
\begin{document}

\title{Anti-Ramsey numbers for cancellative configurations in $p$-graphs}

\author{
    Cheng Chi
    \thanks{School of Mathematical Sciences, Shanghai Jiao Tong University, 800 Dongchuan Road, Shanghai 200240, China.
        Email: chengchi@sjtu.edu.cn.
        Supported by National Key R\&D Program of China under Grant No. 2022YFA1006400 and National Natural Science Foundation of China No. 12571376.
    }\qquad
    Long-Tu Yuan
    \thanks{School of Mathematical Sciences and Shanghai Key Laboratory of PMMP, East China Normal University, 500 Dongchuan Road, Shanghai 200240, P.R.  China.
        Email: ltyuan@math.ecnu.edu.cn. Supported in part by National Natural Science Foundation of China grant 12271169 and Science and Technology Commission
        of Shanghai Municipality (No. 22DZ2229014).
    }
}
\date{}
\maketitle

\begin{abstract}
    We study edge-colorings of the complete $p$-graph on $n$ vertices that contain no three edges $A,B,C$ of distinct colors such that the symmetric difference of $A$ and $B$ is contained in $C$.
    For $p\ge3$ and $n\ge p+1$, we show that every such coloring contains at most $1+\floor{n/p}$ colors and characterize the extremal colorings, generalizing a theorem of Erd\H{o}s, Simonovits and S\'os. 
    When $p=3$, the condition $A\triangle B\subseteq C$ implies $|A\triangle B|=2$, and the three edges necessarily form a copy of $F_4\coloneqq\{abc,abd,bcd\}$ or $F_5\coloneqq\{abc,abd,cde\}$.
    For $n\ge5$, we show that every rainbow $F_5$-free edge-coloring is rainbow cancellative.
    For rainbow $F_4$-free colorings, we construct colorings with $m(n)+1$ colors for all $n\ge4$, where $m(n)$ is the size of a maximum partial Steiner triple system of order $n$ and satisfies $m(n)=n^2/6+O(n)$, improving the linear lower bound by Budden and Stiles. 
    Moreover, for $n=2^s-1$, we obtain $\ar(n,F_4)\ge m(n)+n^2/42+o(n^2)=4n^2/21+o(n^2)$ via a construction based on independent sets in the Grassmann graph.
    We also prove that $\ar(n,F_4)\le (5n^2-8n)/21$ for $n\ge4$,     improving the quadratic coefficient in the upper bound of Budden and Stiles from $1/4$ to $5/21$.
\end{abstract}
\section{Introduction}\label{subsec:intro}

Anti-Ramsey theory asks for the maximum number of colors in an edge-coloring of a complete (hyper)graph that avoids a rainbow copy of a prescribed configuration.
Since the work of Erd\H{o}s, Simonovits and S\'os \cite{erdos1975}, it has been understood that anti-Ramsey problems form a natural counterpart to classical Ramsey theory and are closely related to Tur\'an-type extremal questions.
For graphs this interaction has been studied extensively, but for uniform hypergraphs comparatively fewer exact results are known, especially for small local configurations defined by simple relations among only a few edges.
For two sets $A$ and $B$, write $A\triangle B$ for their symmetric difference.
The cancellative relation $A\triangle B\subseteq C$ is a basic example of this type: it is easy to state and arises naturally in extremal set theory.
The aim of this paper is to study the anti-Ramsey behavior of this pattern.

A $p$-uniform hypergraph (or $p$-graph for brevity) is called {\it cancellative} if it contains no distinct edges $A,B,C$ satisfying $A\triangle B\subseteq C$.
Cancellative $p$-graphs have been studied extensively, and many works concern their extremal behavior, particularly the maximum number of edges in a cancellative hypergraph \cite{bollobas1974,keevash2004,pikhurko2008}.
Thus cancellativity already lies naturally in the scope of hypergraph Tur\'an theory.
Given a family of $p$-graphs $\cH$ and a $p$-graph $G$, we say that $G$ is $\cH$-free if it contains no copy of any $H\in\cH$.
The Tur\'an problem asks for the maximum number of edges in an $n$-vertex $\cH$-free $p$-graph; this maximum is denoted by $\ex(n,\cH)$.
When $p=2$, a graph is cancellative if and only if it is triangle-free, and Mantel's theorem \cite{mantel1907} determines its Tur\'an number.

Since ordinary cancellativity is a classical extremal notion, it is natural to ask for its rainbow analogue.
An {\it edge-colored $p$-graph} $H$ is a $p$-graph together with a surjection $\psi\colon E(H)\to C$, where the elements of $C$ are called {\it colors}.
We write $c(H)$ for the number of colors used in $H$, namely $c(H)=|C|$. For an edge $f\in E(H)$, let $c_H(f)$ denote the color assigned to $f$ (we omit the subscript when the context is clear).
For a sub-$p$-graph $H'\subseteq H$, we say that $H'$ is {\it rainbow} if all of its edges receive pairwise distinct colors under $\psi$. For a family of $p$-graphs $\cF$, we say that $H$ is {\it rainbow $\cF$-free} if it contains no rainbow copy of any $F\in\cF$.

Let $\cF^{(p)}$ denote the family of $p$-graphs consisting of three edges $A,B,C$ such that $A\triangle B\subseteq C$, and let $K_n^{(p)}$ be the complete $p$-graph on $n$ vertices.
For a family of $p$-graphs $\cF$, define $\ar(n,\cF)$ to be the maximum number of colors in a rainbow $\cF$-free edge-coloring of $K_n^{(p)}$. This is the {\it anti-Ramsey number} of $\cF$.
In particular, we call an edge-colored $K_n^{(p)}$ {\it rainbow cancellative} if it is rainbow $\cF^{(p)}$-free.
In this language, our goal is to determine how many colors $K_n^{(p)}$ can support before a rainbow cancellative configuration becomes unavoidable.

The graph case provides the natural benchmark.
When $p=2$, the condition $A\triangle B\subseteq C$ means exactly that $A,B,C$ are the three edges of a triangle.
Hence rainbow cancellative colorings of $K_n$ are precisely edge-colorings with no rainbow triangle, usually called Gallai colorings~\cite{gallai}. The corresponding anti-Ramsey number was determined by Erd\H{o}s, Simonovits and S\'os \cite{erdos1975}.

\begin{theorem}[Erd\H{o}s, Simonovits and S\'os \cite{erdos1975}]\label{thm:ess}
    For every $n\ge3$, $\ar(n,\cF^{(2)})=n-1$.
\end{theorem}

Theorem \ref{thm:ess} completely settles the graph case.
Beyond graphs, related anti-Ramsey-type problems for uniform hypergraphs have been studied mainly for highly structured configurations, such as matchings, paths, and cycles.
Matchings in complete uniform hypergraphs were studied by \"Ozkahya and Young \cite{ozkahya2013}; in the $3$-uniform case, further exact results were obtained by Guo, Lu, and Peng \cite{guo2023}.
For paths and cycles, Gu, Li, and Shi \cite{gu2020} determined exact anti-Ramsey numbers of linear and loose paths and cycles for sufficiently large $n$.
In contrast, small local configurations in $3$-uniform hypergraphs have received much less attention.

Our first main result determines the anti-Ramsey number of rainbow cancellative configurations for all $p\ge3$ and characterizes the extremal colorings.

\begin{theorem}\label{thm:rainbow cancellative}
    If $n$ and $p$ are integers with $n\ge p+1$ and $p\ge3$, then
    \[
        \ar(n,\cF^{(p)})=1+\floor{n/p}.
    \]
    Furthermore, every extremal edge-colored $H \coloneqq K_n^{(p)}$ contains a vertex subset $U$ with $|U|=p\floor{n/p}$ such that $K_n^{(p)}[U]$ contains $\floor{n/p}$ vertex-disjoint rainbow edges, and every other edge of $K_n^{(p)}[U]$ receives one additional common color.
\end{theorem}

Thus the general cancellative problem has an exact linear answer.
When $p=3$ the structure of the forbidden family becomes more concrete, and this makes it possible to go further.
Indeed, if $A,B,C$ are distinct edges in a $3$-graph with $A\triangle B\subseteq C$, then necessarily $|A\triangle B|=2$, and the three triples form a copy of one of the following two $3$-graphs:
\[
    F_4\coloneqq\{abc,abd,bcd\},
    \qquad
    F_5\coloneqq\{abc,abd,cde\}.
\]
Consequently, an edge-colored $3$-graph is rainbow cancellative if and only if it is rainbow $\{F_4,F_5\}$-free.
This leads to a more delicate question: what happens if one forbids only \(F_4\) or only \(F_5\)?

The first surprise is that \(F_5\) is not an independent obstruction in the anti-Ramsey setting.
In a complete \(3\)-graph, excluding rainbow \(F_5\) already excludes rainbow \(F_4\), and hence already forces rainbow cancellativity.

\begin{theorem}\label{thm: F5 free}
    Let $n$ be an integer with $n\ge5$.
    If an edge-colored $K_n^{(3)}$ is rainbow $F_5$-free, then it is rainbow $F_4$-free.
    In particular, it is rainbow cancellative.
\end{theorem}
As an immediate consequence of \cref{thm: F5 free,thm:rainbow cancellative}, we obtain the following exact formula.
\begin{corollary}\label{coro:maximum color}
    If $n$ is an integer with $n\ge5$, then $\ar(n,F_5)=1+\floor{n/3}$.
\end{corollary}

Theorem \ref{thm: F5 free} shows that $F_5$ still behaves in a rather rigid way.
The remaining configuration $F_4$, however, is substantially more difficult.
In contrast to the rainbow cancellative problem and to \(F_5\), the configuration \(F_4\) leads to quadratic behavior.
This is already visible in the ordinary Tur\'an problem.
Frankl and F\"uredi \cite{frankl} constructed $F_4$-free $3$-graphs with Tur\'an density $2/7$, and Mubayi \cite{mubayi} conjectured that $\pi(F_4)\coloneqq \lim_{n\to\infty}\frac{\ex(n,F_4)}{\binom{n}{3}}=2/7$.
The best known upper bound $\pi(F_4)\le 0.286889$, due to Falgas-Ravry and Vaughan \cite{falgasravry}, comes from flag algebras.

The same difficulty appears on the anti-Ramsey side.
Budden and Stiles \cite{budden} initiated the study of $\ar(n,F_4)$ and proved the following bounds.
\begin{theorem}[Budden and Stiles \cite{budden}]\label{thm F4-1}
    If $n\ge5$, then
    \begin{equation*}
        n-2\le \ar(n,F_4)\le
        \begin{cases}
            \frac{1}{4}n(n-2)-1  & \text{if $n$ is even;} \\
            \frac{1}{4}(n-1)^2-1 & \text{if $n$ is odd.}
        \end{cases}
    \end{equation*}
\end{theorem}

Theorem \ref{thm F4-1} leaves a substantial gap between the lower and upper bounds.
To improve the lower bound, one seeks many colors while keeping pair-overlaps among differently colored triples under very tight control.
This points naturally to design-theoretic constructions.
Indeed, a copy of \(F_4\) is created by repeated pair-intersections among triples, so a natural starting point is a large linear \(3\)-graph, that is, a family of triples in which each pair of vertices appears in at most one triple.
This leads directly to Steiner triple systems and their partial variants.

A {\it Steiner triple system} of order $n$, denoted $\STS(n)$, is a pair $(V,\mathcal B)$ where $V$ is a finite set of vertices with $|V|=n$ and $\mathcal B\subseteq \binom{V}{3}$ is a family of $3$-subsets (called {\it blocks}) such that every pair $\{x,y\}\in\binom{V}{2}$ is contained in exactly one block of $\mathcal B$.
A classical theorem in design theory states that an $\STS(n)$ exists if and only if $n\equiv 1$ or $3 \pmod 6$, see \cite{kirkman}.
For general $n$, if we only insist that every pair is used at most once, then it leads to the following notions.
\begin{definition}\label{def:max-psts}
    A {\it partial Steiner triple system} (PSTS) of order $n$ is a pair $(V,\mathcal P)$, where $V$ is a set of $n$ vertices and $\mathcal P\subseteq \binom{V}{3}$ is a family of triples such that every pair $\{x,y\}\in\binom{V}{2}$ is contained in at most one triple in $\mathcal P$.
    A {\it maximum partial Steiner triple system} (MPSTS) of order \(n\), denoted by \(\MPSTS(n)\), is a PSTS of order \(n\) that maximizes the number of triples.\footnote{Readers may consider $\STS$, $\PSTS$ and $\MPSTS$ as $3$-uniform linear hypergraphs.}
\end{definition}

Denote the number of blocks in an $\MPSTS(n)$ by $m(n)$.
The exact value of $m(n)$ is given by the following classical theorem of Sch\"onheim \cite{schonheim1966}.
\begin{theorem}[Sch\"onheim \cite{schonheim1966}]
    For every $n$ with $n\ge3$, we have
    \begin{equation}\label{eq:Dv32-mod6}
        m(n)=
        \begin{cases}
            \frac{n(n-2)}{6},   & n\equiv 0,2 \pmod 6, \\
            \frac{n(n-1)}{6},   & n\equiv 1,3 \pmod 6, \\
            \frac{n^2-2n-2}{6}, & n\equiv 4 \pmod 6,   \\
            \frac{n^2-n-8}{6},  & n\equiv 5 \pmod 6.
        \end{cases}
    \end{equation}
\end{theorem}
In particular, $m(n)=n^2/6+O(n)$, so an $\MPSTS(n)$ already has the correct scale to yield a quadratic lower bound.
Our first construction indeed colors the triples of an $\MPSTS(n)$ with distinct colors and assigns one additional common color to all remaining triples.
For $n=2^s-1$, the situation is richer: the projective-geometric model of $\STS(n)$ provides extra structure, and independent sets in the Grassmann graph allow us to introduce many more colors while still avoiding rainbow $F_4$.
This leads to our third main result.

\begin{theorem}\label{thm:F4}
    If $n\ge4$, then
    \[
        \frac{5n^2-8n}{21} \ge \ar(n,F_4) \ge m(n)+1.
    \]
    Moreover, if $n=2^s-1$ for some $s\ge3$, then
    \[
        \ar(n,F_4)\ge m(n)+\frac{n^2}{42}+o(n^2)=\frac{4}{21}n^2+o(n^2).
    \]
\end{theorem}

The paper is organized as follows.
In \cref{sec:def lem}, we collect notations and auxiliary lemmas.
Sections \ref{sec: proof of 1}, \ref{sec: proof of 2} and \ref{sec: proof of 3} contain the proofs of \cref{thm:rainbow cancellative}, \cref{thm: F5 free}, and \cref{thm:F4}, respectively.
We conclude in \cref{sec:concluding} with a further problem.

\section{Definitions and Lemmas}\label{sec:def lem}

In this section, we will introduce several useful definitions and lemmas.
Let $p,q$ be integers with $0\le q\le p$.
Let $S_{q,r}^{(p)}$ be the $p$-uniform hypergraph with edge set
\[
    E(S_{q,r}^{(p)})=\{Q\cup P_i: i\in[r]\},
\]
where $Q,P_1,\dots,P_r$ are pairwise disjoint vertex sets with $|Q|=q$ and $|P_i|=p-q$ for all $i\in[r]$.
The set $Q$ is called the {\it core} of $S_{q,r}^{(p)}$.

The hypergraph $S_{q,r}^{(p)}$ can be seen as a generalization of the star graph when $q\ge1$, and as a generalization of the matching graph when $q=0$.
Several examples of the hypergraph $S_{q,r}^{(p)}$ are given in Figure \ref{fig:spqr}.
\begin{figure}[H]
    \centering
    \begin{tikzpicture}
        [
            color1/.style={draw, rounded corners, inner sep=6pt},
            color2/.style={draw, very thick, rounded corners, densely dashed, gray, inner sep=14pt},
        ]
        \node (v11) at (-1.2,-1) {};
        \node (v12) at (-1.2,-0.5) {};
        \node (v13) at (-1.2,0) {};
        \node (v14) at (-1.2,0.5) {};
        \node (v15) at (-1.2,1) {};
        \fill (v11) circle (0.1) node {};
        \fill (v12) circle (0.1) node {};
        \fill (v13) circle (0.1) node {};
        \fill (v14) circle (0.1) node {};
        \fill (v15) circle (0.1) node {};
        \node (v21) at (0,-1) {};
        \node (v22) at (0,-0.5) {};
        \node (v23) at (0,0) {};
        \node (v24) at (0,0.5) {};
        \node (v25) at (0,1) {};
        \fill (v21) circle (0.1) node {};
        \fill (v22) circle (0.1) node {};
        \fill (v23) circle (0.1) node {};
        \fill (v24) circle (0.1) node {};
        \fill (v25) circle (0.1) node {};
        \node (v31) at (1.2,-1) {};
        \node (v32) at (1.2,-0.5) {};
        \node (v33) at (1.2,0) {};
        \node (v34) at (1.2,0.5) {};
        \node (v35) at (1.2,1) {};
        \fill (v31) circle (0.1) node {};
        \fill (v32) circle (0.1) node {};
        \fill (v33) circle (0.1) node {};
        \fill (v34) circle (0.1) node {};
        \fill (v35) circle (0.1) node {};
        \node [color1, fit=(v11) (v13) (v15)] {};
        \node [color1, fit=(v21) (v23) (v25)] {};
        \node [color1, fit=(v31) (v33) (v35)] {};
        \node [color2, fit=(v11) (v35)] {};
        \node at (-0,-2) {(i)};
    \end{tikzpicture}
    \quad
    \begin{tikzpicture}
        [
        color1/.style={draw, rounded corners, inner sep=6pt},
        color1'/.style={draw, rounded corners, inner sep=4pt},
        color1''/.style={draw, rounded corners, inner sep=2pt},
        color1'''/.style={draw, rounded corners, inner sep=8pt},
        color2/.style={draw, very thick, rounded corners, densely dashed, gray, inner sep=14pt},
        ]
        \node (c1) at (0.2,0.2) {};
        \node (c2) at (0.2,-0.2) {};
        \node (c3) at (-0.2,-0.2) {};
        \node (c4) at (-0.2,0.2) {};
        \fill (c1) circle (0.1) node {};
        \fill (c2) circle (0.1) node {};
        \fill (c3) circle (0.1) node {};
        \fill (c4) circle (0.1) node {};
        \node (v11) at (1,0.2) {};
        \node (v12) at (1,-0.2) {};
        \node (v21) at (-1,0.2) {};
        \node (v22) at (-1,-0.2) {};
        \node (v31) at (0.2,1) {};
        \node (v32) at (-0.2,1) {};
        \node (v41) at (0.2,-1) {};
        \node (v42) at (-0.2,-1) {};
        \fill (v11) circle (0.1) node {};
        \fill (v12) circle (0.1) node {};
        \fill (v21) circle (0.1) node {};
        \fill (v22) circle (0.1) node {};
        \fill (v31) circle (0.1) node {};
        \fill (v32) circle (0.1) node {};
        \fill (v41) circle (0.1) node {};
        \fill (v42) circle (0.1) node {};
        \node (s1) at (1,1) {};
        \node (s2) at (-1,-1) {};
        \node [color1, fit=(c1) (c3) (v11)] {};
        \node [color1', fit=(c1) (c3) (v21)] {};
        \node [color1'', fit=(c1) (c3) (v31)] {};
        \node [color1''', fit=(c1) (c3) (v41)] {};
        \node [color2, fit=(s1) (s2)] {};
        \node at (-0,-2) {(ii)};
    \end{tikzpicture}
    \quad
    \begin{tikzpicture}
        [
        color1/.style={draw, rounded corners, inner sep=6pt},
        color1'/.style={draw, rounded corners, inner sep=4pt},
        color1''/.style={draw, rounded corners, inner sep=2pt},
        color1'''/.style={draw, rounded corners, inner sep=8pt},
        color2/.style={draw, very thick, rounded corners, densely dashed, gray, inner sep=14pt},
        ]
        \node (c1) at (0.2,0) {};
        \node (c2) at (-0.2,0) {};
        \node (c3) at (0,0.2) {};
        \node (v1) at (0,1) {};
        \node (v2) at (1.2,0) {};
        \node (v3) at (0,-1) {};
        \fill (c1) circle (0.1) node {};
        \fill (c2) circle (0.1) node {};
        \fill (c3) circle (0.1) node {};
        \fill (v1) circle (0.1) node {};
        \fill (v2) circle (0.1) node {};
        \fill (v3) circle (0.1) node {};
        \node (s1) at (0.2,0.2) {};
        \node (s2) at (-0.2,1) {};
        \node [color1, fit=(c1) (c2) (v1)] {};
        \node [color1', fit=(c2) (c3) (v2)] {};
        \node [color1'', fit=(s1) (c2) (v3)] {};
        \node [color2, fit=(s2) (v2) (v3)] {};
        \node at (0.5,-2) {(iii)};
    \end{tikzpicture}
    \caption{(i) $S_{0,3}^{(5)}$, (ii) $S_{4,4}^{(6)}$, and (iii) $S_{3,3}^{(4)}$.}
    \label{fig:spqr}
\end{figure}
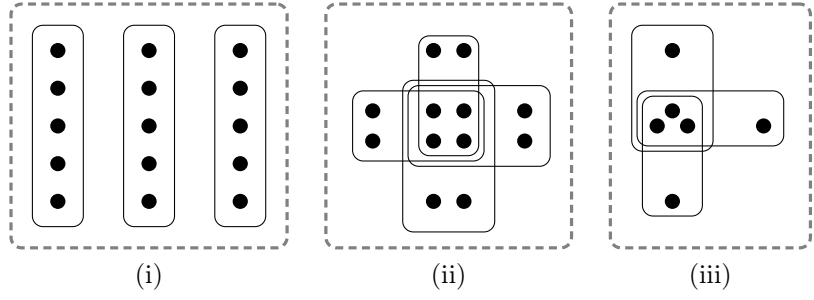

Denote by $\cT^{(p)}$ the minimal\footnote{In the sense that if $H_1\subsetneq H_2$ and $H_2\in\cT^{(p)}$, then $H_1\notin \cT^{(p)}$.} family of $p$-uniform hypergraphs consisting of three edges $f_1,f_2,f_3$ with $\left|\bigcup_{1\le i<j\le 3}f_i\triangle f_j\right|\le p$.
This family of $p$-graphs is crucial in the proof of \cref{thm:rainbow cancellative}, see the following lemma.
\begin{lemma}\label{lem:rainbow triple}
    Let $p$ be an integer with $p\ge3$ and $H$ be an edge-colored complete $p$-graph.
    If $H$ is rainbow cancellative, then it is rainbow $\cT^{(p)}$-free, and hence it is rainbow $S^{(p)}_{p-1,3}$-free.
\end{lemma}

\begin{proof}
    Suppose to the contrary that $H$ contains some rainbow copy of $T$ in $\cT^{(p)}$.
    Let the three edges of $T$ be $f_1,f_2,f_3$.
    Since $\left|\bigcup_{1\le i<j\le 3}f_i\triangle f_j\right|\le p$, there exists an edge $f_4$ of $H$ with $\left(\bigcup_{1\le i<j\le 3}f_i\triangle f_j\right)\subseteq f_4$.
    It follows from $f_1\triangle f_2\subseteq f_4$ and $H$ is rainbow cancellative that $c(f_4)\in\{c(f_1),c(f_2)\}$.
    Similarly, we have $c(f_4)\in\{c(f_1),c(f_3)\}$ and $c(f_4)\in\{c(f_2),c(f_3)\}$.
    This shows that $c(f_4)\in \bigcap_{1\le i<j\le 3}\{c(f_i),c(f_j)\}=\emptyset$, a contradiction.

    To see that $H$ is rainbow $S^{(p)}_{p-1,3}$-free, we only need to prove that $S^{(p)}_{p-1,3}\in \cT^{(p)}$.
    Let the three edges of a $S^{(p)}_{p-1,3}$ be $f_1',f_2',f_3'$.
    By the definition of $S^{(p)}_{p-1,3}$, there is a set $X$ with $|X|=p-1$ and three vertices $x_1,x_2,x_3$ such that $f_i'=X\cup\{x_i\}$.
    It is easy to verify that
    \[
        \left|\bigcup_{1\le i<j\le 3}f_i'\triangle f_j'\right|=\left|\{x_1,x_2,x_3\}\right|=3\le p,
    \]
    and hence $S^{(p)}_{p-1,3}\in \cT^{(p)}$.
    This proves the lemma.
\end{proof}

Denote the $3$-graph $\{abc,abd,abe\}$ by $H_1$ and the $3$-graph $\{abc,bcd,cde\}$ by $H_2$.
The illustrations of $H_1$ and $H_2$ are given in Figure \ref{fig:H1H2}.
\begin{figure}[H]
    \centering
    \begin{tikzpicture}
        [
        color1/.style={draw, rounded corners, inner sep=14pt},
        color1'/.style={draw, rounded corners, inner sep=12pt},
        color1''/.style={draw, rounded corners, inner sep=10pt},
        color2/.style={draw, very thick, rounded corners, densely dashed, gray, inner sep=18pt},
        ]
        \node (v1) at (0.5,0) {};
        \node (v2) at (-0.5,0) {};
        \node (v3) at (0,1) {};
        \node (v4) at (1.5,0) {};
        \node (v5) at (-1.5,0) {};
        \node (r1) at (1.5,1) {};
        \node (r2) at (-1.5,-1) {};
        \fill (v1) circle (0.1) node[below] {$b$};
        \fill (v2) circle (0.1) node[below] {$a$};
        \fill (v3) circle (0.1) node[below] {$c$};
        \fill (v4) circle (0.1) node[below] {$d$};
        \fill (v5) circle (0.1) node[below] {$e$};
        \node [color1'', fit=(v1) (v2) (v3)] {};
        \node [color1', fit=(v1) (v2) (v4)] {};
        \node [color1, fit=(v1) (v2) (v5)] {};
        \node [color2, fit=(r1) (r2)] {};
        \node at (-0,-2.1) {(i)};
    \end{tikzpicture}
    \quad
    \begin{tikzpicture}
        [
        color1/.style={draw, rounded corners, inner sep=14pt},
        color1'/.style={draw, rounded corners, inner sep=12pt},
        color1''/.style={draw, rounded corners, inner sep=10pt},
        color2/.style={draw, very thick, rounded corners, densely dashed, gray, inner sep=18pt},
        ]
        \node (v1) at (-2,0) {};
        \node (v2) at (-1,0) {};
        \node (v3) at (0,0) {};
        \node (v4) at (1,0) {};
        \node (v5) at (2,0) {};
        \node (r1) at (2,1) {};
        \node (r2) at (-2,-1) {};
        \fill (v1) circle (0.1) node[below] {$a$};
        \fill (v2) circle (0.1) node[below] {$b$};
        \fill (v3) circle (0.1) node[below] {$c$};
        \fill (v4) circle (0.1) node[below] {$d$};
        \fill (v5) circle (0.1) node[below] {$e$};
        \node [color1, fit=(v1) (v2) (v3)] {};
        \node [color1', fit=(v2) (v3) (v4)] {};
        \node [color1'', fit=(v3) (v4) (v5)] {};
        \node [color2, fit=(r1) (r2)] {};
        \node at (-0,-2.1) {(ii)};
    \end{tikzpicture}
    \caption{(i) $H_1$, and (ii) $H_2$.}
    \label{fig:H1H2}
\end{figure}
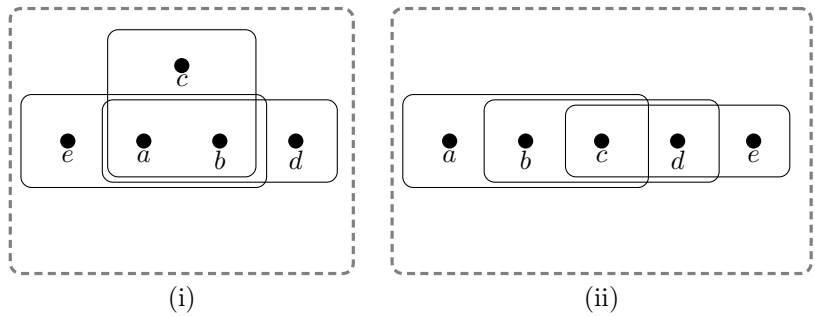
\begin{lemma}\label{lem:F5 free}
    If $H$ is an edge-colored complete $3$-graph, then $H$ is rainbow $F_5$-free if and only if $H$ is rainbow $\{H_1,H_2\}$-free.
\end{lemma}

\begin{proof}
    Assume that $H$ is rainbow $\{H_1,H_2\}$-free.
    We now prove that $H$ is rainbow $F_5$-free.
    Suppose to the contrary that $H$ contains a rainbow copy of $F_5$ with edges $uvx,uvy,xyz$.
    Since $H$ is rainbow $H_2$-free, if $c(uxy)\notin\{c(xyz),c(uvy)\}$, then $xyz,uxy,uvy$ form a rainbow $H_2$, a contradiction.
    Similarly, if $c(uxy)\notin\{c(xyz),c(uvx)\}$, then $xyz,uxy,uvx$ form a rainbow $H_2$.
    This proves that $c(uxy)=c(xyz)$.
    Since $H$ is rainbow $H_1$-free, if $c(uvz)\notin\{c(uvx),c(uvy)\}$, then $uvx,uvy,uvz$ form a rainbow $H_1$, a contradiction.
    However, $uxy,uvx,uvz$ form a rainbow $H_2$ if $c(uvz)=c(uvy)$ and $uxy,uvy,uvz$ form a rainbow $H_2$ if $c(uvz)=c(uvx)$, see \cref{fig:H1H2 to F5}.
    This shows that $H$ is rainbow $F_5$-free.
    \begin{figure}[H]
        \begin{center}
            \qquad
            \begin{tikzpicture}
                [               scale=1.0,
                color1/.style={draw, very thick, rounded corners, red, inner sep=12pt},
                color1'/.style={draw, very thick, rounded corners, red, inner sep=4pt},
                color2/.style={draw, very thick, densely dashed, rounded corners, blue, inner sep=10pt},
                color3/.style={draw, very thick, densely dotted, rounded corners, violet, inner sep=8pt},
                color4/.style={draw, very thick, dash dot, rounded corners, teal, inner sep=8pt},
                ]
                \node (z) at (-2,0) {};
                \node (v) at (-1,0) {};
                \node (u) at (0,0) {};
                \node (x) at (1,1) {};
                \node (y) at (1,-1) {};
                \fill (z) circle (0.1) node[below] {$z$};
                \fill (v) circle (0.1) node[below] {$v$};
                \fill (u) circle (0.1) node[below] {$u$};
                \fill (x) circle (0.1) node[below] {$x$};
                \fill (y) circle (0.1) node[below] {$y$};
                \node [color1, fit=(u) (v) (x)] {};
                \node [color2, fit=(u) (v) (y)] {};
                \node [color3, fit=(u) (x) (y)] {};
                \node [color4, fit=(u) (v) (z)] {};
                \draw[very thick, rounded corners, red, inner sep=1pt] (2.4,2.2) rectangle (3.1,2.5);
                \node (legend0) at (3.7,2.35) {$c(uvx)$};
                \draw[very thick, densely dashed, rounded corners, blue, inner sep=1pt] (2.4,1.7) rectangle (3.1,2.0);
                \node (legend1) at (3.7,1.85) {$c(uvy)$};
                \draw[very thick, densely dotted, rounded corners, violet, inner sep=1pt] (2.4,1.2) rectangle (3.1,1.5);
                \node (legend2) at (3.7,1.35) {$c(xyz)$};
                \draw[very thick, dash dot, rounded corners, teal, inner sep=1pt] (-3.5,2.2) rectangle (-2.8,2.5);
                \node (legend4) at (-1.3,2.35) {$c(uvx)$ or $c(uvy)$};
            \end{tikzpicture}
        \end{center}
        \caption{}
        \label{fig:H1H2 to F5}
    \end{figure}

    Assume that $H$ is rainbow $F_5$-free.
    We now prove that $H$ is rainbow $\{H_1,H_2\}$-free.
    Suppose that $H$ contains a rainbow copy of $H_1$ with edges $f_1\coloneqq abc,f_2\coloneqq abd$ and $f_3\coloneqq abe$.
    Suppose that the color of $cde$ is $\gamma$.
    If $\gamma=c(f_i)$, say $i=1$, for some $i$, then $f_2,f_3,cde$ form a rainbow copy of $F_5$, a contradiction.
    If $\gamma\ne c(f_i)$ for $i=1,2,3$, then $f_2,f_3,cde$ again form a rainbow copy of $F_5$.
    Hence $H$ is rainbow $H_1$-free.

    Suppose that $H$ contains a rainbow copy of $H_2$ with edges $e_1 \coloneqq x_1x_2x_3$, $e_2\coloneqq x_2x_3x_4$ and $e_3\coloneqq  x_3x_4x_5$.
    Let us establish the following claim to get a contradiction.
    \begin{claim}\label{clm:5 cycle}
        We have $c(x_4x_5x_1)=c(e_1)$.
    \end{claim}
    \begin{proof}
        Denote the edge $x_4x_5x_1$ by $e_4$.
        If $c(e_4)\notin\{c(e_1),c(e_3)\}$, then $e_1,e_3,e_4$ form a rainbow $F_5$, a contradiction.
        If $c(e_4)\notin\{c(e_1),c(e_2)\}$, then $e_1,e_2,e_4$ form a rainbow $F_5$, a contradiction again.
        This shows that $c(e_4)=c(e_1)$ and proves the claim.
    \end{proof}
    \cref{clm:5 cycle} shows that $c(x_4x_5x_1)=c(e_1)$, and hence $x_2x_3x_4,x_3x_4x_5,x_4x_5x_1$ form a rainbow copy of $H_2$.
    Applying \cref{clm:5 cycle} to this rainbow $H_2$ implies that $c(x_5x_1x_2)=c(x_2x_3x_4)$, and hence $x_3x_4x_5,x_4x_5x_1,x_5x_1x_2$ form a rainbow $H_2$.
    Applying \cref{clm:5 cycle} again shows that $c(e_1)=c(e_3)$, which is impossible since $e_1,e_2,e_3$ form a rainbow $H_2$, see \cref{fig:5 cycle}.
\end{proof}
\begin{figure}[H]
    \centering
    \qquad\qquad\qquad
    \begin{tikzpicture}
        [               scale=0.85,
        color1/.style={draw, very thick, rounded corners, red, inner sep=12pt},
        color1'/.style={draw, very thick, rounded corners, red, inner sep=4pt},
        color2/.style={draw, very thick, densely dashed, rounded corners, blue, inner sep=10pt},
        color3/.style={draw, very thick, densely dotted, rounded corners, violet, inner sep=8pt},
        color4/.style={draw, thick, dash dot, rounded corners, yellow, inner sep=4pt},
        ]
        \node (x11) at (-4.5,2) {};
        \node (x21) at (-4.5,1) {};
        \node (x31) at (-4.5,0) {};
        \node (x41) at (-4.5,-1) {};
        \node (x51) at (-4.5,-2) {};
        \fill (x11) circle (0.1) node[below] {$x_1$};
        \fill (x21) circle (0.1) node[below] {$x_2$};
        \fill (x31) circle (0.1) node[below] {$x_3$};
        \fill (x41) circle (0.1) node[below] {$x_4$};
        \fill (x51) circle (0.1) node[below] {$x_5$};
        \node [color1, fit=(x11) (x21) (x31)] {};
        \node [color2, fit=(x21) (x31) (x41)] {};
        \node [color3, fit=(x31) (x41) (x51)] {};
        \node (x22) at (-1.5,2) {};
        \node (x32) at (-1.5,1) {};
        \node (x42) at (-1.5,0) {};
        \node (x52) at (-1.5,-1) {};
        \node (x12) at (-1.5,-2) {};
        \fill (x22) circle (0.1) node[below] {$x_2$};
        \fill (x32) circle (0.1) node[below] {$x_3$};
        \fill (x42) circle (0.1) node[below] {$x_4$};
        \fill (x52) circle (0.1) node[below] {$x_5$};
        \fill (x12) circle (0.1) node[below] {$x_1$};
        \node [color1, fit=(x42) (x52) (x12)] {};
        \node [color2, fit=(x22) (x32) (x42)] {};
        \node [color3, fit=(x32) (x42) (x52)] {};
        \node (x33) at (1.5,2) {};
        \node (x43) at (1.5,1) {};
        \node (x53) at (1.5,0) {};
        \node (x13) at (1.5,-1) {};
        \node (x23) at (1.5,-2) {};
        \fill (x23) circle (0.1) node[below] {$x_2$};
        \fill (x33) circle (0.1) node[below] {$x_3$};
        \fill (x43) circle (0.1) node[below] {$x_4$};
        \fill (x53) circle (0.1) node[below] {$x_5$};
        \fill (x13) circle (0.1) node[below] {$x_1$};
        \node [color1, fit=(x43) (x53) (x13)] {};
        \node [color2, fit=(x53) (x13) (x23)] {};
        \node [color3, fit=(x33) (x43) (x53)] {};
        \node (x44) at (4.5,2) {};
        \node (x54) at (4.5,1) {};
        \node (x14) at (4.5,0) {};
        \node (x24) at (4.5,-1) {};
        \node (x34) at (4.5,-2) {};
        \fill (x24) circle (0.1) node[below] {$x_2$};
        \fill (x34) circle (0.1) node[below] {$x_3$};
        \fill (x44) circle (0.1) node[below] {$x_4$};
        \fill (x54) circle (0.1) node[below] {$x_5$};
        \fill (x14) circle (0.1) node[below] {$x_1$};
        \node [color1, fit=(x44) (x54) (x14)] {};
        \node [color2, fit=(x54) (x14) (x24)] {};
        \node [color3, fit=(x14) (x24) (x34)] {};
        \node (a1) at (-3,0) {$\Longrightarrow$};
        \node (a2) at (0,0) {$\Longrightarrow$};
        \node (a3) at (3,0) {$\Longrightarrow$};
        \draw[very thick, rounded corners, red, inner sep=1pt] (5.5,2.2) rectangle (6.2,2.5);
        \node (legend0) at (6.7,2.35) {$c(e_1)$};
        \draw[very thick, densely dashed, rounded corners, blue, inner sep=1pt] (5.5,1.7) rectangle (6.2,2.0);
        \node (legend1) at (6.7,1.85) {$c(e_2)$};
        \draw[very thick, densely dotted, rounded corners, violet, inner sep=1pt] (5.5,1.2) rectangle (6.2,1.5);
        \node (legend2) at (6.7,1.35) {$c(e_3)$};
    \end{tikzpicture}
    \caption{}
    \label{fig:5 cycle}
\end{figure}

\section{Proof of \cref{thm:rainbow cancellative}}\label{sec: proof of 1}

Let $n$ and $p$ be integers with $n\ge p+1$ and $p\ge3$.
Let the vertex set of $H' \coloneqq K_n^{(p)}$ be $V=\{v_1,\ldots,v_n\}$.
Let us color the edges of $H'$ by the following rules:
\[
    c(v_{ip+1}v_{ip+2}\ldots v_{(i+1)p})=i+1
\]
for $i=0,\ldots,\floor{n/p}-1$ and $c(S)=\floor{n/p}+1$ for the remaining $p$-tuples $S$.
It is easy to see that $H'$ is rainbow cancellative under the coloring defined above, and hence
\begin{equation}\label{eq:lower bound}
    \ar(n,\cF^{(p)})\ge 1+\floor{n/p}.
\end{equation}

Let $H$ be an $n$-vertex edge-colored complete $p$-graph with $c(H)=\ar(n,\cF^{(p)})$.
In the rest of this section, we will prove that $c(H)=1+\floor{n/p}$ and there exists a vertex subset $U$ with $|U|=p\cdot \floor{n/p}$ such that $H[U]$ contains $\floor{n/p}$ pairwise vertex-disjoint edges that receive pairwise distinct colors, and the remaining edges in $H[U]$ are all colored by one additional color.

To this end, let us first establish several claims.

\medskip

\begin{claim}
    $H$ contains a rainbow copy of $S_{p-1,2}^{(p)}$.
\end{claim}

\begin{proof}
    If the claim does not hold, then
    \begin{equation}\label{eq:assumption 1}
        c(S\cup \{s\})=c(S\cup \{s'\})
    \end{equation}
    for every $S\subseteq V$ with $|S|=p-1$ and every $s,s'\in V\setminus S$.

    If $f_1 \coloneqq x_1,\ldots,x_p$ and $f_2 \coloneqq y_1,\ldots,y_p$ are two distinct edges with
    \[
        |\{x_1,\ldots,x_p\}\cap\{y_1,\ldots,y_p\}|=i
    \]
    for some $i$ with $0\le i\le p-1$, then we may assume that $f_1=t_1\cdots t_i s_{i+1}\cdots s_p$ and $f_2=t_1\cdots t_i s_{i+1}'\cdots s_p'$.
    By \cref{eq:assumption 1}, we have
    \[
        c(t_1\cdots t_{i} s_{i+1}\cdots s_p)=c(t_1\cdots t_{i} s_{i+1}\cdots s_p')=\cdots=c(t_1\cdots t_{i} s_{i+1}'\cdots s_p').
    \]
    This shows that $c(f_1)=c(f_2)$, and hence $c(H)=1<1+\floor{n/p}$, which contradicts \cref{eq:lower bound}.
    This completes the proof of the claim.
\end{proof}

Let us choose a rainbow copy of $S_{p-1,2}^{(p)}$ and let the two edges in this rainbow $S_{p-1,2}^{(p)}$ be $X\cup\{t_1\}$ and $X\cup\{t_2\}$, where $X$ is a set with $|X|=p-1$.
Denote $c(X\cup\{t_i\})$ by $\alpha_i$ for $i=1,2$.
By \cref{lem:rainbow triple}, $H$ is rainbow $S_{p-1,3}^{(p)}$-free, and hence the set $V(H)\setminus X$ can be partitioned into two sets $A_1,A_2$ such that $c(X\cup\{x_i\})=\alpha_i$ for every $x_i\in A_i$ and $i=1,2$.
In the rest of the proof, we will prove several claims concerning the colors of some edges.

\begin{claim}\label{clm:crossing color}
    If $f$ is an edge with $f\cap A_1\ne\emptyset$ and $f\cap A_2\ne\emptyset$, then $f$ is colored by $\alpha_1$ or $\alpha_2$.
\end{claim}

\begin{proof}
    Let $f$ be an edge with $f\cap A_1\ne\emptyset$ and $f\cap A_2\ne\emptyset$.
    For $i=1,2$, let $a_i$ be a vertex in $f\cap A_i$ and let $f_i=X\cup\{a_i\}$.
    It follows from the definition of $A_i$ that $c(f_i)=\alpha_i$.
    It is easy to see that $f_1\triangle f_2=\{a_1,a_2\}\subseteq f$.
    Since $H$ is rainbow cancellative, we have $c(f)\in \{\alpha_1,\alpha_2\}$.
    This proves the claim.
\end{proof}

\begin{claim}\label{clm:X two colors}
    If $f$ is an edge with $X\cap f\ne\emptyset$, then $f$ is colored by $\alpha_1$ or $\alpha_2$.
\end{claim}

\begin{proof}
    Let $f_0$ be an edge with $X\cap f_0\ne\emptyset$.
    Denote $c(f_0)$ by $\alpha_0$.
    Denote $f_0\cap X$ by $X_1$ and denote $|X_1|$ by $s$.
    Let us prove that $\alpha_0\in \{\alpha_1,\alpha_2\}$ by induction on $|X\setminus f_0|$, that is, $p-1-s$.
    If $f_0\cap A_1\ne\emptyset$ and $f_0\cap A_2\ne\emptyset$, then the result follows from \cref{clm:crossing color}.
    Hence, we only need to consider the case where one of $f_0\cap A_1$ and $f_0\cap A_2$ is empty.
    Without loss of generality, suppose that $f_0\cap A_2=\emptyset$, and hence $f_0\subseteq A_1\cup X$.

    Let $a_1$ be a vertex in $A_1\cap f_0$ and $a_2$ be a vertex in $A_2$.
    Let $f_i=X\cup\{a_i\}$ for $i=1,2$.
    It follows directly from the definition of $A_i$ that $c(f_i)=\alpha_i$ for $i=1,2$.
    Note that $a_1\in A_1\cap f_0$, and it is easy to see that
    \[
        \bigcup_{0\le i<j\le 2}f_i\triangle f_j=(X\setminus X_1)\cup (f_0\cap A_1)\cup\{a_2\}.
    \]
    Hence, if $s\ge p/2$, then $|X\setminus f_0|=p-1-s\le (p-2)/2$, and hence
    \[
        \left|\bigcup_{0\le i<j\le 2}f_i\triangle f_j\right|=2p-2|X_1|\le p.
    \]
    If $\alpha_0\notin\{\alpha_1,\alpha_2\}$, then $f_0,f_1,f_2$ form a rainbow copy of some $T\in \cT^{(p)}$, which contradicts \cref{lem:rainbow triple}.
    This shows that $\alpha_0=\alpha_i$ for some $i=1,2$ when $s\ge p/2$, and hence proves the basic step of the induction.

    Now let us assume that $s\le (p-1)/2$, and hence $|X\setminus f_0|\ge (p-1)/2$.
    By induction, we suppose that
    \begin{equation}\label{equ:suppose}
        \text{every edge $f$ with $|X\setminus f|\le p-2-s$ is colored by $\alpha_1$ or $\alpha_2$.}
    \end{equation}
    Recall that $X_1=f_0\cap X$ and $|X_1|=s\le (p-1)/2$.
    Let $X_2$ be a subset of $X\setminus X_1$ with $|X_2|=|X_1|$ and $X_3=X\setminus(X_1\cup X_2)$.
    Let $C_1$ be a subset of $(f_0\cap A_1)\setminus\{a_1\}$ with $|C_1|=p-1-2s$ and denote
    \[
        (f_0\cap A_1)\setminus(\{a_1\}\cup C_1)
    \]
    by $C_2$.
    It is clear that $|C_1|=p-1-2s$ and $|C_2|=s$.

    If $s=(p-1)/2$, then $p$ is odd and $C_1=X_3=\emptyset$.
    Note that $f_0=X_1\cup C_2\cup\{a_1\}$ and $f_i=X_1\cup X_2\cup\{a_i\}$ for $i=1,2$ in this case and $c(f_0)=\alpha_0$.
    Suppose to the contrary that $\alpha_0\notin\{\alpha_1,\alpha_2\}$.
    Note that $c(f_1)=\alpha_1$ and $f_0\triangle f_1=X_2\cup C_2$.
    Denote the edge $X_2\cup C_2\cup\{a_2\}$ by $f^*$.
    Since $H$ is rainbow cancellative, the color of $f^*$ is either $\alpha_0$ or $\alpha_1$.
    Since $f^*\cap A_1\ne\emptyset$ and $f^*\cap A_2\ne\emptyset$, by \cref{clm:crossing color}, the color of $f^*$ is either $\alpha_1$ or $\alpha_2$, and hence $c(f^*)=\alpha_1$.
    Note that $c(f_2)=\alpha_2$ and
    \[
        f_2\triangle f^*=X_1\cup C_2\subseteq f_0,
    \]
    which contradicts the choice of $H$, see Figure~\ref{fig:half half}.
    Hence, we have $\alpha_0=\alpha_i$ for some $i\in\{1,2\}$.
    \begin{figure}[H]
        \centering
        \qquad\qquad\quad
        \begin{tikzpicture}
            [
            color1/.style={draw, very thick, rounded corners, red, inner sep=10pt},
            color1'/.style={draw, very thick, rounded corners, red, inner sep=4pt},
            color2/.style={draw, very thick, densely dashed, rounded corners, blue, inner sep=8pt},
            color3/.style={draw, very thick, densely dotted, rounded corners, violet, inner sep=6pt},
            color4/.style={draw, thick, dash dot, rounded corners, yellow, inner sep=4pt},
            ]

            \node (X1) at (-1,0) {$X_1$};
            \node (X2) at (1,0) {$X_2$};
            \node (a1) at (-2,0) {$a_1$};
            \node (a2) at (2,0) {$a_2$};
            \node (C2) at (0,1) {$C_2$};
            \node [color1, fit=(X1) (X2) (a1)] {};
            \node [color2, fit=(X1) (X2) (a2)] {};
            \node [color3, fit=(X1) (C2) (a1)] {};
            \node [color1', fit=(X2) (C2) (a2)] {};
            \draw[very thick, rounded corners, red, inner sep=10pt] (3.5,1.2) rectangle (4.2,1.5);
            \node (legend0) at (4.5,1.35) {$\alpha_0$};
            \draw[very thick, densely dashed, rounded corners, blue, inner sep=1pt] (3.5,0.7) rectangle (4.2,1.0);
            \node (legend1) at (4.5,0.85) {$\alpha_1$};
            \draw[very thick, densely dotted, rounded corners, violet, inner sep=1pt] (3.5,0.2) rectangle (4.2,0.5);
            \node (legend2) at (4.5,0.35) {$\alpha_2$};
        \end{tikzpicture}
        \caption{}
        \label{fig:half half}
    \end{figure}

    Now let us assume that $s\le (p-2)/2$, and hence $p-1-2s\ge1$ and $X_3\ne\emptyset$.
    Recall that $f_0=X_1\cup C_1\cup C_2\cup\{a_1\}$ and $f_i=X_1\cup X_2\cup X_3\cup\{a_i\}$ for $i=1,2$ in this case.
    Suppose to the contrary that $\alpha_0\notin\{\alpha_1,\alpha_2\}$.
    Let $x_1$ be a vertex in $X_1$ and $x_2$ be a vertex in $X_2$.
    Denote the edge $(f_0\cup\{a_2\})\setminus\{x_1\}$ by $g_1$
    and $(f_0\cup\{x_2\})\setminus\{a_1\}$ by $g_2$.
    Note that $g_1\triangle f_0=\{x_1,a_2\}\subseteq f_2$ and $c(f_2)=\alpha_2$.
    Since $H$ is rainbow cancellative, we have $c(g_1)=\alpha_0$ or $c(g_1)=\alpha_2$.
    By \cref{clm:crossing color}, we conclude that $c(g_1)\in\{\alpha_1,\alpha_2\}$, and hence $c(g_1)=\alpha_2$.
    Since $g_2\triangle f_0=\{x_2,a_1\}\subseteq f_1$, we have $c(g_2)\in\{\alpha_0,\alpha_1\}$.
    Note that $|X\setminus g_2|=p-2-s$, so by the induction hypothesis \eqref{equ:suppose}, we have $c(g_2)\in\{\alpha_1,\alpha_2\}$, and therefore $c(g_2)=\alpha_1$.
    Since $1\le s\le (p-2)/2$, we have $p\ge4$.
    It is easy to verify that
    \[
        |(g_1\triangle g_2) \cup (g_1\triangle f_0) \cup (g_2 \triangle f_0)| = |\{x_1,x_2,a_1,a_2\}| = 4 \le p,
    \]
    and hence $f_0,g_1,g_2$ form a rainbow copy of some $T\in \cT^{(p)}$, which contradicts \cref{lem:rainbow triple}.
    This shows that $c(f_0)\in\{\alpha_1,\alpha_2\}$.
\end{proof}

\begin{remark}\label{rem:core-two-colors}
    The proof of \cref{clm:X two colors} only uses that $X$ is a $(p-1)$-set for which there exist
    two edges $X\cup\{t_1\},X\cup\{t_2\}$ of distinct colors.
    Hence the same conclusion holds with $X$ replaced by the core of any rainbow copy of $S_{p-1,2}^{(p)}$.
    We will use this fact repeatedly.
\end{remark}
For an integer $i$ with $1\le i\le p-1$, we say the property $\PP_i$ is satisfied if $H[A_1]\cup H[A_2]$ does not contain rainbow $S^{(p)}_{p-i,2}$ whose edges all have colors distinct from $\alpha_1$ and $\alpha_2$.

\begin{claim}\label{clm:disjoint}
    For $i=1,\ldots, p-1$, the property $\PP_i$ is satisfied.
\end{claim}
\begin{proof}
    We will prove the claim by induction on $i$.
    When $i=1$, if the property $\PP_i$ is not satisfied, then $H[A_1]\cup H[A_2]$ contains a rainbow copy of $S^{(p)}_{p-1,2}$ whose edges are $f_3,f_4$ such that $\{c(f_3),c(f_4)\}\cap \{\alpha_1,\alpha_2\}=\emptyset$.
    Let $x$ be a vertex in $X$ and let $x'$ be a vertex in $f_3\cap f_4$.
    By applying \cref{clm:X two colors} to $X$ and $f_3\cap f_4$ respectively, we conclude that every edge incident to $\{x,x'\}$ is colored by $\{c(f_3),c(f_4)\}\cap \{\alpha_1,\alpha_2\}=\emptyset$, which is impossible.
    This proves the basic step for the induction.

    Let us assume that $i\ge2$ and the property $\PP_{j}$ is satisfied for every $j$ with $j\le i-1$.
    If the property $\PP_{i}$ is not satisfied, then $H[A_1]\cup H[A_2]$ has a rainbow copy of $S^{(p)}_{p-i,2}$ whose edges are $f_3=x_1\cdots x_{p-i} y_{p-i+1}\cdots y_p$ and $f_4=x_1\cdots x_{p-i} z_{p-i+1}\cdots z_p$ such that $\{c(f_3),c(f_4)\}\cap \{\alpha_1,\alpha_2\}=\emptyset$.
    Let us define
    \begin{align}
        f_3'=x_1\cdots x_{p-i-1}z_{p-i+1}y_{p-i+1}\cdots y_p
    \end{align}
    and
    \[
        f_4'=x_1\cdots x_{p-i-1}y_{p-i+1}z_{p-i+1}\cdots z_{p}
    \]
    It is easy to verify that $f_3'\triangle f_3=\{x_{p-i},z_{p-i+1}\}\subseteq f_4$, and hence $c(f_3')\in\{c(f_3),c(f_4)\}$.
    It can be checked easily that $|f_3\cap f_3'|=p-1$.
    Since the property $\PP_{1}$ is satisfied, we have $c(f_3')=c(f_3)$.
    Similarly, we conclude that $c(f_4')=c(f_4)$.
    Note that
    \begin{align}
        f_3'\cap f_4'=\{x_1,\ldots,x_{p-i-1},y_{p-i+1},z_{p-i+1}\},
    \end{align}
    and hence $|f_3'\cap f_4'|=p-(i-1)$.
    Note that $c(f_3')=c(f_3),c(f_4')=c(f_4)$ and $\{c(f_3),c(f_4)\}\cap \{\alpha_1,\alpha_2\}=\emptyset$.
    Hence, the property $\PP_{i-1}$ is not satisfied, which contradicts the induction hypothesis.
    This completes the proof of the induction and the proof of the claim.
\end{proof}

\begin{claim}\label{clm:one special}
    There exists some $i\in\{1,2\}$ such that every edge in $H[A_i]$ is colored by $\alpha_1$ or $\alpha_2$.
\end{claim}
\begin{proof}
    Suppose to the contrary that for $i=1,2$, there exists an edge $h_i\in H[A_i]$ with $\{c(h_i)\}\cap\{\alpha_1,\alpha_2\}=\emptyset$.
    For $i=1,2$, denote $c(h_i)$ by $\beta_i$ and let $b_i$ be a vertex in $h_i\cap A_i$.
    Denote $h_i\setminus\{b_i\}$ by $B_i$.
    Let $y$ be a vertex in $X$ and denote $X\setminus\{y\}$ by $Y$.
    Let $Y'$ be a subset of $X$ with $y\in Y'$ and $|Y'|=p-2$.

    Recall that $A_i$ is defined by $A_i=\{z\in V(H)\setminus X;c(X\cup\{z\})=\alpha_i\}$, and hence $c(Y\cup\{y,b_i\})=\alpha_i$.
    Note that
    \[
        (\{y\}\cup B_i) \triangle (\{b_i\}\cup B_i)=\{y,b_i\}\subseteq (Y\cup\{y,b_i\}),
    \]
    and hence $\{y\}\cup B_i$ is colored by $\alpha_i$ or $\beta_i$.
    Since $y\in X$, by \cref{clm:X two colors}, we conclude that $c(\{y\}\cup B_i)=\alpha_i$.
    Denote the edge $Y'\cup\{b_1,b_2\}$ by $h_3$.
    For $i=1,2$, note that $c(\{y\}\cup B_i)=\alpha_i$, $c(\{b_i\}\cup B_i)=\beta_i$ and
    \[
        (\{y\}\cup B_i) \triangle (\{b_i\}\cup B_i)=\{y,b_i\}\subseteq h_3.
    \]
    If $\beta_1=\beta_2$, then $c(h_3)\in \{\alpha_1,\beta_1\}\cap\{\alpha_2,\beta_2\}=\{\beta_1\}$, and so $c(h_3)=\beta_1$, which contradicts \cref{clm:X two colors}.
    If $\beta_1\ne\beta_2$, then $c(h_3)\in \{\alpha_1,\beta_1\}\cap\{\alpha_2,\beta_2\}=\emptyset$, which is impossible, see \cref{fig:only one useful}.\footnote{We omit $Y'$ and $h_3$ in the figure since their inclusion would significantly complicate the drawing.}
    This completes the proof of the claim.
\end{proof}
\begin{figure}[H]
    \centering
    \qquad\qquad\qquad
    \begin{tikzpicture}
        [
        color1/.style={draw, very thick, rounded corners, red, inner sep=2pt},
        color1'/.style={draw, very thick, rounded corners, red, inner sep=-2pt},
        color2/.style={draw, very thick, densely dashed, rounded corners, blue, inner sep=0pt},
        color2'/.style={draw, very thick, densely dashed, rounded corners, blue, inner sep=6pt},
        color2''/.style={draw, very thick, densely dashed, rounded corners, blue, inner sep=4pt},
        color3/.style={draw, very thick, densely dotted, rounded corners, violet, inner sep=0pt},
        color4/.style={draw, very thick, dash dot, rounded corners, teal, inner sep=2pt},
        ]

        \node (y) at (-1,-1) {$y$};
        \node (Y) at (0,0) {$Y$};
        \node (b1) at (-2,0) {$b_1$};
        \node (b2) at (1,0) {$b_2$};
        \node (B1) at (-1,1) {$B_1$};
        \node (B2) at (1,-2) {$B_2$};
        \node [color1, fit=(y) (Y) (b1)] {};
        \node [color2, fit=(y) (Y) (b2)] {};
        \node [color3, fit=(b1) (B1)] {};
        \node [color4, fit=(b2) (B2)] {};
        \node [color1', fit=(y) (B1)] {};
        \node [color2'', fit=(y) (B2)] {};
        \draw[very thick, rounded corners, red, inner sep=10pt] (2.5,1.2) rectangle (3.2,1.5);
        \node (legend1) at (3.5,1.35) {$\alpha_1$};
        \draw[very thick, densely dashed, rounded corners, blue, inner sep=1pt] (2.5,0.7) rectangle (3.2,1.0);
        \node (legend2) at (3.5,0.85) {$\alpha_2$};
        \draw[very thick, densely dotted, rounded corners, violet, inner sep=1pt] (2.5,0.2) rectangle (3.2,0.5);
        \node (legend3) at (3.5,0.35) {$\beta_1$};
        \draw[very thick, dash dot, rounded corners, teal, inner sep=1pt] (2.5,-0.3) rectangle (3.2,0);
        \node (legend4) at (3.5,-0.15) {$\beta_2$};
    \end{tikzpicture}
    \caption{}
    \label{fig:only one useful}
\end{figure}
By \cref{clm:one special}, we may and do assume that every edge in $H[A_1]$ is colored by $\alpha_1$ or $\alpha_2$.
By \cref{clm:crossing color} and \cref{clm:X two colors}, we have
\[
    c(H)\le 2+k,
\]
where $k$ denotes the number of colors other than $\alpha_1,\alpha_2$ in the edge-colored graph $H[A_2]$.
Note that $|A_1|\ge1$ and $|X|=p-1$, and hence $|A_2|= n-|A_1|-|X|\le n-p$.
By \cref{clm:disjoint}, we conclude that $k\le \floor{|A_2|/p}\le \floor{(n-p)/p}$.
This proves that
\[
    c(H)\le 2+k\le 1+\floor{n/p}.
\]
Furthermore, it is easy to see from the proof that $c(H)=1+\floor{n/p}$ only when $H[A_2]$ contains $\floor{(n-p)/p}$ vertex-disjoint rainbow edges which are not colored by $\alpha_1$ or $\alpha_2$.

In the rest of the proof, we will show that $H$ contains a vertex subset $U$ with $|U|=p\cdot\floor{n/p}$ such that $H[U]$ can be obtained by taking $\floor{n/p}$ vertex-disjoint rainbow edges and coloring all the remaining edges with one additional color.
Now let us suppose that $n=pq+r$, where $q=\floor{n/p}$ and $r$ is an integer with $0\le r\le p-1$.
Let the $\floor{(n-p)/p}$ vertex-disjoint rainbow edges which are not colored by $\alpha_1$ or $\alpha_2$ in $H[A_2]$ be $g_1,\ldots,g_{q-1}$.
For $i\in[q-1]$, denote $c(g_i)$ by $\beta_i$, where $\beta_i\notin\{\alpha_1,\alpha_2\}$ and $\beta_1,\ldots,\beta_{q-1}$ are pairwise distinct.

Let $z$ be a fixed vertex in $A_1$.
Denote the edge $X\cup\{z\}$ by $g_q$ and the set $\bigcup_{i=1}^q g_{i}$ by $U$
In the rest of the proof, we show that every edge $h$ in $H[U]$ with $h\ne g_i$ for $i\in[q]$ is colored by $\alpha_2$.

For $i\in[q-1]$ and $y\in g_i$, denote the edge $g_i\cup\{z\}\setminus\{y\}$ by $g_i(y)$.
By \cref{clm:crossing color}, the color of $g_i(y)$ is either $\alpha_1$ or $\alpha_2$.
If there exists distinct $y,y'$ in $g_i$ such that $c(g_i(y))\ne c(g_i(y'))$, then the colors of $g_i(y),g_i(y')$ and $g_i$ are pairwise distinct and
\[
    |(g_i(y)\triangle g_i(y')) \cup (g_i(y) \triangle g_i) \cup (g_i(y')\triangle g_i)| = |\{y,y',z\}|=3 \le p,
\]
contradicting \cref{lem:rainbow triple}.
This shows that $c(g_i(y))=c(g_i(y'))$ for all $y,y'\in g_i$ and $i\in[q-1]$.
Note that every $g_i$ and $g_i(y)$ form a rainbow copy of $S_{p-1,2}^{(p)}$.
Applying \cref{clm:X two colors} to this rainbow copy of $S_{p-1,2}^{(p)}$ implies every edge $h$ that intersects $g_i\cap g_i(y)=g_i\setminus\{y\}$ is colored by $c(g_i)$ or $c(g_i(y))$.
Particularly, the edge $X\cup\{y'\}$ is colored by $\alpha_2$ by the definition of $A_2$, and hence $c(g_i(y))=\alpha_2$, where $y'$ is a vertex in $g_i\setminus\{y\}$.

For an edge $h\in H[\bigcup_{i=1}^{q-1} g_i]$ with $h\ne g_k$ for all $k\in[q-1]$, there exists $g_i$ and $g_j$ such that $h\cap g_i\ne\emptyset$ and $h\cap g_j\ne\emptyset$.
Let $y_i$ be a vertex in $g_i\setminus h$ and $y_j$ be a vertex in $g_j\setminus h$.
The edge $h$ intersects both $g_i(y_i)\cap g_i$ and $g_j(y_j)\cap g_j$.
Applying \cref{clm:X two colors} gives
\[
    c(h)\in \{c(g_i),c(g_i(y_i))\}\cap \{c(g_j),c(g_j(y_j))\}=\{\alpha_2\},
\]
and hence $c(h)=\alpha_2$.

Let $h$ be an edge in $H[U]$ with $h\cap g_q\ne\emptyset$ and $h\ne g_q$.
Clearly, either $h\cap X\ne\emptyset$ or $z\in h$ and $h\cap A_2\ne\emptyset$ holds.
By \cref{clm:X two colors} and \cref{clm:crossing color}, the color of $h$ is either $\alpha_1$ or $\alpha_2$.
Since $h\ne g_q$, $h$ intersects an edge $g_i$ for some $i\in[q-1]$.
Let $y$ be a vertex in $g_i\setminus h$.
Previous arguments imply every edge, particularly $h$, that intersects $g_i\setminus \{y\}$ is colored by $\alpha_2$ or $c(g_i)$.
Therefore, we have
\[
    c(h)\in\{\alpha_1,\alpha_2\}\cap \{\alpha_2,c(g_i)\}=\{\alpha_2\}.
\]
This shows $c(h)=\alpha_2$.
Therefore, every edge $h\in H[U]$ with $h\neq g_r$ for all $r\in[q]$ is colored by $\alpha_2$.

\section{Proof of \cref{thm: F5 free} and \cref{coro:maximum color}}\label{sec: proof of 2}

Denote the $3$-graph $\{abc,abd,abe\}$ by $H_1$ and $\{abc,bcd,cde\}$ by $H_2$. 
Denote the $3$-graph $\{abc,abd,bcd\}$ by $F_4$ and $\{abc,abd,cde\}$ by $F_5$.

\medskip

\noindent {\bf Proof of \cref{thm: F5 free}:}
Let $n$ be an integer with $n\ge5$.
Suppose that $H$ is an edge-colored rainbow $F_5$-free $K_n^{(3)}$ that contains a rainbow $F_4$.
Let $\{xyz,xyw,xzw\}$ be a rainbow $F_4$ in $H$.
Denote $c(xyz)$, $c(xyw)$ and $c(xzw)$ by $\alpha_1,\alpha_2$ and $\alpha_3$, respectively.
Denote $V(H)$ by $V$ and $V(H)\setminus\{x,y,z,w\}$ by $W$.

Let $a$ be a vertex in $W$.
By \cref{lem:F5 free}, $H$ is rainbow $\{H_1,H_2\}$-free.
Note that $\alpha_1,\alpha_2$ and $\alpha_3$ are pairwise distinct.
If $c(xya)\notin\{\alpha_1,\alpha_2\}$, then $xya,xyz,xyw$ form a rainbow copy of $H_1$.
If $c(xya)\notin\{\alpha_2,\alpha_3\}$, then $xya,xyw,xzw$ form a rainbow copy of $H_2$, and hence $c(xya)=\alpha_2$.
This shows that $xya,xyz,xzw$ form a rainbow $H_2$, which contradicts \cref{lem:F5 free}.

\begin{figure}[H]
    \centering
    \qquad
    \begin{tikzpicture}
        [
        color1/.style={draw, very thick, rounded corners, red, inner sep=12pt},
        color1'/.style={draw, very thick, rounded corners, red, inner sep=6pt},
        color2/.style={draw, very thick, densely dashed, rounded corners, blue, inner sep=10pt},
        color2'/.style={draw, very thick, densely dashed, rounded corners, blue, inner sep=8pt},
        color2''/.style={draw, very thick, densely dashed, rounded corners, blue, inner sep=4pt},
        color3/.style={draw, very thick, densely dotted, rounded corners, violet, inner sep=8pt},
        color3'/.style={draw, very thick, densely dotted, rounded corners, violet, inner sep=6pt},
        color4/.style={draw, very thick, dash dot, rounded corners, teal, inner sep=6pt},
        ]
        \node (x1) at (0,0) {};
        \node (y1) at (1,0) {};
        \node (z1) at (-1,0) {};
        \node (w1) at (0,1) {};
        \node (a1) at (0,-1) {};
        \fill (x1) circle (0.1) node[below] {$x$};
        \fill (y1) circle (0.1) node[below] {$y$};
        \fill (z1) circle (0.1) node[below] {$z$};
        \fill (w1) circle (0.1) node[below] {$w$};
        \fill (a1) circle (0.1) node[right] {$a$};
        \node [color1, fit=(x1) (y1) (z1)] {};
        \node [color2', fit=(x1) (y1) (a1)] {};
        \node [color2, fit=(x1) (y1) (w1)] {};
        \node [color3, fit=(x1) (z1) (w1)] {};
        \draw[very thick, rounded corners, red, inner sep=10pt] (2.5,1.2) rectangle (3.2,1.5);
        \node (legend0) at (3.5,1.35) {$\alpha_1$};
        \draw[very thick, densely dashed, rounded corners, blue, inner sep=1pt] (2.5,0.7) rectangle (3.2,1.0);
        \node (legend1) at (3.5,0.85) {$\alpha_2$};
        \draw[very thick, densely dotted, rounded corners, violet, inner sep=1pt] (2.5,0.2) rectangle (3.2,0.5);
        \node (legend2) at (3.5,0.35) {$\alpha_3$};
    \end{tikzpicture}
    \caption{}
\end{figure}

\medskip

\noindent {\bf Proof of \cref{coro:maximum color}:}
By \cref{thm: F5 free}, when $n\ge5$, every edge-colored rainbow $F_5$-free $K_n^{(3)}$ is rainbow cancellative, and hence
\[
    \ar(n,F_5)\le \ar(n,\cF^{(3)}).
\]
Since $\cF^{(3)}=\{F_4,F_5\}$, every rainbow $\cF^{(3)}$-free edge-coloring is also rainbow $F_5$-free, and hence
\[
    \ar(n,F_5)\ge \ar(n,\cF^{(3)}).
\]
This proves
\[
    \ar(n,F_5)=\ar(n,\cF^{(3)})=1+\floor{n/3}.
\]

\section{Proof of \cref{thm:F4}}\label{sec: proof of 3}
\subsection{Construction of the lower bound}

Let us first prove $\ar(n,F_4)\ge m(n)+1$ for all $n\ge4$
Let $P\coloneqq K_n^{(3)}$ be an edge-colored complete $3$-graph.
For a vertex $x\in V(P)$, define an edge-colored graph $P(x)\coloneqq K_{n-1}^{(2)}$ on vertex set
$V(P(x))\coloneqq V(P)\setminus\{x\}$ by setting
\[
    c_{P(x)}(yz)\coloneqq c_P(xyz)\qquad\text{for all distinct }y,z\in V(P(x)).
\]
It is straightforward to check that $P$ is rainbow $F_4$-free if and only if $P(x)$ is rainbow $K_3^{(2)}$-free for every $x\in V(P)$.

A {\it representing graph} of an edge-colored graph $G$ is a spanning subgraph of $G$ obtained by choosing one edge from each color class.
It is easy to verify that $G$ is rainbow $K_3^{(2)}$-free if and only if every representing graph of $G$ is $K_3^{(2)}$-free.

For general $n$ with $n\ge4$, let $H$ be an edge-colored $K_n^{(3)}$ obtained from a rainbow maximum partial Steiner triple system (MPSTS) of order $n$ by coloring all other edges with one extra color $\alpha$.
Let $v$ be a vertex of $H$ and let $L$ be a representing graph of $H(v)$.
The edges of $L$ coming from blocks containing $v$ form a matching, and $L$ contains exactly one additional edge of color $\alpha$.
Hence $L$ is triangle-free, so $H(v)$ is rainbow $K_3^{(2)}$-free.
Since $v$ was arbitrary, $H$ is rainbow $F_4$-free.
It is easy to check that $c(H)=m(n)+1$, and hence $\ar(n,F_4)\ge m(n)+1$.
This completes the proof.

\bigskip

Now we prove the stronger bound for $n=2^s-1$.
Fix an integer $s\ge3$ and put
\[
    n\coloneqq 2^s-1,\qquad V\coloneqq \FF_2^s\setminus\{0\}.
\]
We will construct a rainbow $F_4$-free edge-coloring of $K_n^{(3)}$ using
\[
    m(n)+\frac{n^2}{42}+o(n^2)
\]
colors.

Given a positive integer $t$ and a subset $S\subseteq \FF_2^s\setminus\{0\}$, we say that $S$ is a {\it projective $(t-1)$-dimensional subspace} of $\FF_2^s\setminus\{0\}$ if $S\cup\{0\}$ is a linear $t$-dimensional subspace of $\FF_2^s$.
Projective $0$-, $1$-, and $2$-dimensional subspaces are called {\it projective points}, {\it projective lines}, and {\it projective planes}, respectively.

Let $G\coloneqq \PG(s-1,2)$ be the Steiner triple system with vertex set $V$ and block set
\[
    E(G)\coloneqq \bigl\{\{x,y,x+y\}\colon x,y\in\FF_2^s\setminus\{0\},x\ne y\bigr\}.
\]
Equivalently, the vertices and blocks of $\PG(s-1,2)$ are the projective points and projective lines in $\FF_2^s\setminus\{0\}$, respectively.

\medskip
For distinct $x,y,z\in \FF_2^s\setminus\{0\}$, define
\begin{equation}\label{eq:def of cl}
    \langle x,y,z\rangle \coloneqq \text{the linear span of } \{x,y,z\} \text{ in }\FF_2^s
    \text{ and }
    \Pi(x,y,z) \coloneqq \langle x,y,z\rangle\setminus\{0\}.
\end{equation}
The following proposition summarizes the basic properties of $\Pi(x,y,z)$ that we will use later.
\begin{proposition}\label{prop:property of cl}
    If $x,y,z\in V(G)$ are distinct, then the following statements hold. \\
    \noindent (i) $xyz\in E(G)$ if and only if $\Pi(x,y,z)$ is a projective line of $\FF_2^s\setminus\{0\}$. \\
    \noindent (ii) $xyz\notin E(G)$ if and only if $\Pi(x,y,z)$ is a projective plane of $\FF_2^s\setminus\{0\}$.
    Moreover, $\Pi(x,y,z)$ is the unique projective plane containing $\{x,y,z\}$ and $\Pi(x,y,z)$ is isomorphic to $\PG(2,2)$, the unique Steiner triple system on seven vertices.\footnote{$\PG(2,2)$ is also called the {\it Fano plane}.} \\
    \noindent (iii) If $xyz\in E(G)$ and $d$ is a vertex in $V(G)\setminus\{x,y,z\}$, then
    \[
        \Pi(x,y,d)=\Pi(x,z,d)=\Pi(y,z,d).
    \]
\end{proposition}
\begin{proof}
    (i) We have $xyz\in E(G)$ if and only if $z=x+y$, equivalently, $x,y,z$ are linearly dependent over $\FF_2$.
    Since $x,y$ are distinct and nonzero, $\langle x,y,z\rangle$ equals the linear span of $x,y$.
    Therefore, $\langle x,y,z\rangle$ is a linear $2$-dimensional subspace of $\FF_2^s$, and hence $\Pi(x,y,z)=\{x,y,x+y\}$ is a projective line.

    Conversely, if $\Pi(x,y,z)$ is a projective line, then $\langle x,y,z\rangle$ is a $2$-dimensional subspace of $\FF_2^s$, hence $x,y,z$ are linearly dependent.
    Over $\FF_2$, any dependence among three distinct nonzero vectors forces $z=x+y$, so $xyz\in E(G)$.

    (ii)  The first assertion holds since it is the contrapositive of (i).
    For uniqueness, let $S$ be any projective plane containing $\{x,y,z\}$.
    Clearly, $S\cup\{0\}$ is a linear $3$-dimensional subspace of $\FF_2^s$, and since $x,y,z\in S\cup\{0\}$ we have $\langle x,y,z\rangle\subseteq S\cup\{0\}$.
    Both $\langle x,y,z\rangle$ and $S\cup\{0\}$ are $3$-dimensional, so $S\cup\{0\}=\langle x,y,z\rangle$ and hence $S=\Pi(x,y,z)$.
    Finally, $\Pi(x,y,z)$ has size $2^3-1=7$ and it is easy to check that $u+v\in\Pi(x,y,z)$ for every distinct $u,v\in\Pi(x,y,z)$, so it forms an $\STS(7)$; the unique $\STS(7)$ is $\PG(2,2)$ (see, e.g., Remark~2.19 in Section~2.4 of~\cite{handbook}).

    (iii) Since $xyz\in E(G)$, by the definition of $\PG(s-1,2)$, we have $x+y=z$, equivalently, $x,y,z$ are linearly dependent over $\FF_2$.
    Therefore, we have
    \[
        \langle x,y,d \rangle=\langle x,z,d \rangle=\langle y,z,d \rangle.
    \]
    Removing $\{0\}$ gives the desired result.
\end{proof}

\medskip

A key definition used in the proof is the Grassmann graph \cite{chow1949}.
For $k\ge2$, the {\it Grassmann graph} $J_q(s,k)$ is a graph whose vertex set is the family of projective $(k-1)$-dimensional subspaces of $\FF_q^s\setminus\{0\}$.\footnote{The Grassmann graph was defined on linear spaces originally and we use its projective form to simplify the proof.}
Two vertices are adjacent if and only if the corresponding projective subspaces intersect in a projective $(k-2)$-dimensional subspace.

In this subsection we put
\[
    J\coloneqq J_2(s,3).
\]
Given three distinct projective planes $W_1,W_2$ and $W_3$, if they form a triangle in $J$, then $W_1\cap W_2\cap W_3$ is either a projective point or a projective line.
For a triangle $W_1W_2W_3$, we call it a {\it point-type triangle} if $W_1\cap W_2\cap W_3$ is a projective point and call it a {\it line-type triangle} otherwise.
A vertex-coloring $\varphi \colon V(J)\to\cC$ is called {\it good} if $J$ contains no point-type triangle whose three vertices receive pairwise distinct colors.
Write $\tau(J)$ for the maximum possible $|\cC|$ over all good colorings of $J$. 

Fix a good-vertex coloring $\varphi\colon V(J)\to[r]$ with $r=\tau(J)$.
For $i=1,\ldots,r$, suppose that $V_i$ is the set of projective planes colored by $i$-th color under $\varphi$.
Let $\beta_1,\ldots,\beta_r$ be pairwise distinct colors, all distinct from the block colors introduced below.

Let $H\coloneqq K_n^{(3)}$ on vertex set $V$.
Define an edge-coloring of $H$ by:
\begin{itemize}
    \item every block $e\in E(G)$ receives its own distinct {\it block color} $\alpha_e$;
    \item for each $i\in[r]$ and each projective plane $W\in V_i$, color every $T$ in $\binom{W}{3}\setminus E(G)$ by the {\it background color} $\beta_i$.
\end{itemize}
By \cref{prop:property of cl}(ii), every non-block triple $xyz$ lies in the unique projective plane $\Pi(x,y,z)$, hence the above coloring is well-defined.

The next proposition records a local constraint that motivates assigning a single background color to all non-block triples inside a projective plane.
\begin{proposition}\label{prop:fano_comp_mono}
    Let $H$ be an edge-colored rainbow $F_4$-free $K_n^{(3)}$.
    Let $W$ be a rainbow copy of $\PG(2,2)$ in $H$.
    If the color of each block-edge of $W$ is used exactly once, then all triple $T$ in $\binom{V(W)}{3}\setminus E(W)$ have the same color in $H$.
\end{proposition}

\begin{proof}
    Since $H$ is rainbow $F_4$-free, for every $4$-set $Q\subseteq V(H)$ the four triples in $\binom{Q}{3}$ use at most two colors: otherwise three of them would have pairwise distinct colors and form a rainbow copy of $F_4$ on $Q$.

    Fix $x\in W$.
    In the Fano plane $W$, the three blocks through $x$ partition $W\setminus\{x\}$ into three disjoint pairs.
    Write these blocks as
    \[
        \{x,a_1,b_1\},\qquad \{x,a_2,b_2\},\qquad \{x,a_3,b_3\}.
    \]
    We first show that all non-block triples in $\binom{W}{3}\setminus E(W)$ that contain $x$ have the same color.

    Consider the $4$-set $Q=\{x,a_1,b_1,a_2\}$.
    It contains the block $\{x,a_1,b_1\}$, whose color is used exactly once in $H$.
    Hence none of the other three triples in $\binom{Q}{3}$ can have this block-color.
    Since $\binom{Q}{3}$ uses at most two colors, the three non-block triples $xa_1a_2,xb_1a_2,a_1b_1a_2$ all have the same color.
    In particular,
    \begin{equation}\label{eq:fano_cross_1}
        c_H(xa_1a_2)=c_H(xb_1a_2).
    \end{equation}
    Applying the same argument to $Q'=\{x,a_2,b_1,b_2\}$ and $Q''=\{x,a_1,b_1,b_2\}$ gives
    \begin{equation}\label{eq:fano_cross_2}
        c_H(xb_1a_2)=c_H(xb_1b_2) \text{ and }c_H(xa_1b_2)=c_H(xb_1b_2)
    \end{equation}
    respectively, and hence
    \[
        c_H(xa_1a_2)=c_H(xa_1b_2)=c_H(xb_1a_2)=c_H(xb_1b_2).
    \]
    Take $u\in\{a_1,b_1\}$ and $v\in\{a_2,b_2\}$.
    Let $w$ be the third vertex of the unique block in $W$ containing the pair $\{u,v\}$.
    The point $w$ cannot lie in $\{x,a_1,b_1,a_2,b_2\}$ since each pair in an $\STS(7)$ lies in a unique block, and hence $w\in\{a_3,b_3\}$.
    Now consider the $4$-set $\{x,u,v,w\}$, which contains the block $\{u,v,w\}$.
    Similar argument above implies
    \[
        c_H(xuv)=c_H(xuw)=c_H(xvw).
    \]
    Hence the non-block colors coming from the pairs of blocks
    \[
        (\{x,a_i,b_i\},\{x,a_j,b_j\})
    \]
    are all equal, where $1\le i<j\le3$.
    Therefore, there exists a color $\gamma(x)$ such that $c_H(xyz)=\gamma(x)$ for all non-block triple $xyz\in \binom{V(W)}{3}\setminus E(W)$.

    Finally, we show $\gamma(x)$ is independent of $x$.
    Take any two distinct points $x,y\in W$ and let $\{x,y,z\}\in E(W)$ be the unique block containing $\{x,y\}$.
    Pick $d\in W\setminus\{x,y,z\}$.
    The $4$-set $\{x,y,z,d\}$ contains the block $xyz$, whose color is used exactly once in $H$.
    Hence the three remaining triples $xyd,xzd,yzd$ cannot use this block-color and must share one common color.
    We have $c_H(xyd)=\gamma(x)$ and $c_H(yzd)=\gamma(y)$, so $\gamma(x)=\gamma(y)$.
    Since $x,y$ were arbitrary, all $\gamma(x)$ are equal to a single color $\gamma_W$.
    This proves the desired result.
\end{proof}

Denote the resulting edge-colored $K_n^{(3)}$ by $\wt{H}$.
We now verify that $\wt{H}$ contains no rainbow copy of $F_4$.
\begin{lemma}\label{lem:auxiliary}
    Let $G$ be $\PG(s-1,2)$ for some $s\ge3$ and let $\wt{H}$ be defined as above.
    Then $\wt{H}$ is rainbow $F_4$-free.
\end{lemma}
\begin{proof}
    Suppose for a contradiction that $\wt{H}$ contains a rainbow copy of $F_4$ on vertices $\{x,y,z,w\}$ with edges $xyz,xyw,yzw$.
    Since $G$ is an $\STS$, among these three triples at most one is a block of $G$ as otherwise two distinct blocks would contain the same pair.

    \medskip
    \noindent{\bf Case 1: none of $xyz,xyw,yzw$ is a block of $G$.}

    In this case, each of $xyz,xyw,yzw$ receives a background color.
    Put
    \[
        P_1\coloneqq \Pi(x,y,z),
        P_2\coloneqq \Pi(x,y,w)\text{ and }
        P_3\coloneqq \Pi(y,z,w).
    \]
    By \cref{prop:property of cl}(ii), $P_1,P_2,P_3$ are all projective planes.
    If two of them are equal, say $P_1=P_2$, then both $xyz$ and $xyw$ lie in the same projective plane and hence receive the same background color, contradicting that this copy of $F_4$ is rainbow.
    Therefore $P_1,P_2,P_3$ are pairwise distinct.

    Each pair among $P_1,P_2,P_3$ intersects in a projective line:
    $P_1\cap P_2$ contains the line through $x$ and $y$, $P_1\cap P_3$ contains the line through $y$ and $z$, and $P_2\cap P_3$ contains the line through $y$ and $w$.
    These three lines are distinct and all pass through $y$, so
    \[
        P_1\cap P_2\cap P_3=\{y\},
    \]
    a projective point.
    Hence $P_1P_2P_3$ is a point-type triangle in $J$.

    Since the coloring $\varphi$ is good, in graph $J$, there is no point-type triangle whose three vertices receive distinct colors.
    Hence, two of $P_1,P_2,P_3$, say $P_1$ and $P_2$, receive the same color.
    Therefore, we have $c(xyz)=c(xyw)$, contradicting that the copy of $F_4$ is rainbow.

    \medskip
    \noindent{\bf Case 2: exactly one of $xyz,xyw,yzw$ is a block of $G$.}

    Without loss of generality, assume that $xyz\in E(G)$.
    Since $G$ is a Steiner triple system, $xyw$ and $yzw$ are non-block triples and are colored by some background colors.
    By Proposition~\ref{prop:property of cl}(iii), applied to the block $xyz$ and the vertex $w$, we have
    \[
        \Pi(x,y,w)=\Pi(x,z,w)=\Pi(y,z,w).
    \]
    In particular, $\Pi(x,y,w)=\Pi(y,z,w)$, so both $xyw$ and $yzw$ lie in the same projective plane.
    By the coloring rule, all non-block triples inside the same projective plane receive the same background color, so $c(xyw)=c(yzw)$, contradicting that $xyz,xyw,yzw$ have pairwise distinct colors.

    \medskip
    In both cases we reach a contradiction. Hence $\wt H$ contains no rainbow copy of $F_4$, as desired.
\end{proof}

By \cref{lem:auxiliary}, the above construction is rainbow $F_4$-free and uses exactly $|E(G)|+r=m(n)+\tau(J)$ colors.
Therefore,
\begin{equation}\label{eq:ar-lower-by-tau}
    \ar(n,F_4)\ge m(n)+\tau(J).
\end{equation}

It remains to estimate $\tau(J)$.

\begin{lemma}\label{lem:tau_J23}
    Let $s\to\infty$ and let $J=J_2(s,3)$ be the Grassmann graph on $3$-dimensional subspaces of $\FF_2^s$.
    Write $n\coloneqq 2^s-1$.
    Then
    \[
        \tau(J)\ge\frac{(2^s-1)^2}{42}+o\bigl((2^s-1)^2\bigr).
    \]
\end{lemma}

\begin{proof}
    Write $\alpha(J)$ for the {\it independence number} of $J$, that is, the size of a maximum independent set in $J$.
    Color the vertices of a maximum independent set of $J$ with distinct colors, and color all remaining vertices with one additional color (if any remain).
    Clearly, this vertex coloring contains no triangles whose vertices receive distinct colors.
    Therefore, it is a good coloring, so $\tau(J)\ge \alpha(J)$.
    It remains to estimate $\alpha(J)$.

    The following result is an immediate specialization of a much more general theorem of Blackburn and Etzion on constant-dimension codes; we state only the form needed here.
    \begin{theorem}[Blackburn and Etzion, Theorem 1 in \cite{blackburn}]\label{thm:BE-special}
        As $s\to\infty$, we have
        \[
            \alpha(J_2(s,3))=\frac{(2^s-1)^2}{42}+o((2^s-1)^2).
        \]
    \end{theorem}
    This completes the proof.
\end{proof}

\subsection{Proof of the upper bound}

Throughout this subsection, all hypergraphs are complete $3$-graphs unless otherwise specified.
We prove the following upper bound, which improves the last assertion of \cref{thm:F4}.

For an edge-colored graph $Q$, let $\rho(Q)$ denote the number of colors whose color class contains a monochromatic triangle.
We shall use the following strengthening of the Erd\H{o}s-Simonovits-S\'os theorem for Gallai colorings.

\begin{lemma}\label{lem:gallai-rho}
    Let $Q$ be an edge-colored complete graph with no rainbow triangle. Then
    \begin{equation}\label{eq:gallai-rho}
        c(Q)+\rho(Q)\le |V(Q)|-1.
    \end{equation}
\end{lemma}

\begin{proof}
    We prove the lemma by induction on $|V(Q)|$.
    The assertion is trivial when $|V(Q)|\le2$.
    Let $|V(Q)|\ge3$.

    By Gallai's theorem \cite{gallai}, there is a non-trivial partition $V(Q)=V_1\cup\cdots\cup V_t$ with $t\ge2$ such that every pair of parts is joined by edges of a single color, and all edges between distinct parts use at most two colors in total.

    For each $i\in[t]$, let $Q_i=Q[V_i]$.  By induction,
    \[
        c(Q_i)+\rho(Q_i)\le |V_i|-1.
    \]

    Let $X$ be the set of colors appearing on edges between distinct parts.
    Then $|X|\le2$.  Let
    \[
        S=\bigcup_{i=1}^t \{\text{colors appearing in }Q_i\}
    \]
    and let
    \[
        R=\bigcup_{i=1}^t
        \{\text{colors whose color class contains a monochromatic triangle in }Q_i\}.
    \]
    Thus
    \[
        |S|\le\sum_{i=1}^t c(Q_i),
        \qquad
        |R|\le\sum_{i=1}^t \rho(Q_i).
    \]

    Put $q=|X\setminus S|$ and
    \[
        r=
        \left|
        \left\{
        \gamma\in X\setminus R:
        \text{ the $\gamma$-colored edges of $Q$ contain a monochromatic triangle}
        \right\}
        \right|.
    \]
    Then
    \begin{equation}\label{eq:gallai-rho-reduction}
        c(Q)+\rho(Q)
        \le
        \sum_{i=1}^t\bigl(c(Q_i)+\rho(Q_i)\bigr)+q+r.
    \end{equation}
    It remains to prove
    \begin{equation}\label{eq:q-plus-r}
        q+r\le t-1.
    \end{equation}

    Consider the reduced complete graph on vertex set $[t]$, where the edge
    $ij$ receives the unique color used between $V_i$ and $V_j$.

    If $t=2$, then there is only one cross-color.  If this color is new, it
    contributes one to $q$ and cannot contribute to $r$, since it appears only
    between the two parts.  If it is not new, it contributes nothing to $q$
    and at most one to $r$.  Hence $q+r\le1=t-1$.

    Suppose $t=3$.  If $q=0$, then $q+r\le |X|\le2=t-1$.  If $q=2$, then both
    cross-colors are new.  Since the reduced graph is a triangle using two
    colors, neither of these new colors forms a monochromatic triangle in the
    reduced graph; and because they are new, neither can form a monochromatic
    triangle using two vertices in one part.  Thus $r=0$ and $q+r=2$.
    It remains to consider $q=1$.  Let $\gamma$ be the unique new cross-color.
    If $|X|=1$, then the reduced triangle is monochromatic of color $\gamma$,
    so $\gamma$ may contribute to both $q$ and $r$, but then $q+r\le2$.  If
    $|X|=2$, then $\gamma$ cannot form a monochromatic triangle in the reduced
    triangle, and since it is new, it cannot form a monochromatic triangle
    using two vertices in one part.  Hence only the other cross-color can
    contribute to $r$, so again $q+r\le2$.  Therefore \eqref{eq:q-plus-r}
    holds for $t=3$.

    Suppose $t=4$.  Since $|X|\le2$, we have $q\le2$ and $r\le2$.  Thus the
    only possible violation of \eqref{eq:q-plus-r} would be $q+r=4$.  This
    would mean that both cross-colors are new and both are counted by
    $\rho(Q)$.  Since they are new, each would have to span a monochromatic
    triangle in the reduced $K_4$.  But two triples of a $4$-set always share
    an edge, so two distinct colors cannot both span monochromatic triangles
    in the same reduced $K_4$.  Hence $q+r\le3=t-1$.

    If $t\ge5$, then
    \[
        q+r\le |X|+|X|\le4\le t-1.
    \]

    Combining \eqref{eq:gallai-rho-reduction}, the induction hypothesis, and
    \eqref{eq:q-plus-r}, we obtain
    \[
        c(Q)+\rho(Q)
        \le
        \sum_{i=1}^t(|V_i|-1)+(t-1)
        =
        |V(Q)|-1.
    \]
\end{proof}

We now prove the desired upper bound.

\begin{theorem}\label{thm:asymptotic-all-n}
    For every $n\ge4$,
    \[
        \ar(n,F_4)\le \frac{5n^2-8n}{21}.
    \]
\end{theorem}

\begin{proof}
    Let $G$ be a rainbow $F_4$-free edge-coloring of $K_n^{(3)}$.

    The proof has two main steps.  We first derive a global upper bound by
    applying \cref{lem:gallai-rho}.
    We then derive a complementary lower bound by isolating the singleton-colored triples and using the structure they impose on the remaining colors.
    Comparing these two bounds yields the desired estimate.

    For every color $\lambda$ in $G$, define its support by
    \[
        V_\lambda=\{x\in V(G)\colon \text{there is an edge } e \text{ with } x\in e \text{ and } c_G(e)=\lambda\}.
    \]
    Put
    \[
        I=\sum_\lambda |V_\lambda|,
    \]
    where the summation ranges over all colors in $G$.
    Thus $I$ counts vertex-color incidences: a color with support size $r$ contributes $r$, not merely one.
    This weighting is useful because a color with large support already gives many units toward the final count.

    For every vertex $v\in V(G)$, let $L_v$ be the edge-colored link graph on vertex set $V(G)\setminus\{v\}$, where
    \[
        c_{L_v}(xy)=c_G(vxy)
        \qquad\text{for all distinct }x,y\in V(G)\setminus\{v\}.
    \]
    A rainbow triangle in $L_v$ would correspond exactly to the three triples of a rainbow $F_4$ containing $v$.
    Since $G$ is rainbow $F_4$-free, $L_v$ contains no rainbow triangle.
    By \cref{lem:gallai-rho},
    \[
        c(L_v)+\rho(L_v)\le n-2.
    \]
    Define
    \[
        \rho=\sum_{v\in V(G)}\rho(L_v).
    \]
    Then
    \begin{equation}\label{eq:I-rho-upper}
        I+\rho
        =
        \sum_{v\in V(G)}c(L_v)+\sum_{v\in V(G)}\rho(L_v)
        \le n(n-2).
    \end{equation}

    Let $F$ be the $3$-graph consisting of all singleton-colored triples in $G$, that is,
    \[
        E(F)=\{e\in E(G)\colon c_G(e)\text{ is used exactly once in }G\}.
    \]
    We first note that $F$ is a PSTS.
    Indeed, suppose that two triples of $F$ share a pair, say $abc,abd\in E(F)$.
    Let their colors be $\delta_1$ and $\delta_2$.  Since both colors are used
    exactly once and the two triples are distinct, $\delta_1\ne\delta_2$.
    The triple $acd$ cannot have color $\delta_1$ or $\delta_2$, since either
    choice would repeat a singleton color.
    If $c_G(acd)\notin\{\delta_1,\delta_2\}$, then $abc,abd,acd$ form a rainbow copy of $F_4$, a contradiction.
    Thus no two triples of
    $F$ share a pair, and $F$ is a PSTS.

    We shall use the forcing observation repeatedly.
    It is the basic local mechanism of the proof: a singleton triple sharing a pair with a non-singleton triple forces the two remaining triples on the same $4$-set to take the non-singleton color.

    \begin{claim}\label{clm:forcing-new}
        Let $xyz$ be a triple of a non-unique color $\alpha$.  Suppose that
        $xyd\in E(F)$.  Then
        \[
            c_G(xzd)=c_G(yzd)=\alpha.
        \]
    \end{claim}

    \begin{proof}
        Let $\delta=c_G(xyd)$.  The color $\delta$ is used exactly once.
        Consider $xzd$.  It cannot have color $\delta$, because this would
        repeat the singleton color of $xyd$.  If $c_G(xzd)\notin
            \{\alpha,\delta\}$, then the three triples $xyz,xyd,xzd$ form a rainbow copy of $F_4$ on the $4$-set $\{x,y,z,d\}$, impossible.
        Hence $c_G(xzd)=\alpha$.  The proof for $yzd$ is identical.
    \end{proof}

    Write $s=e(F)$.
    Let $L(F)$ be the {\it leave graph} of $F$, namely the graph on $V(G)$ whose
    edges are the pairs not contained in any triple of $F$, and write
    \[
        \ell=e(L(F)).
    \]
    The leave graph records precisely the pairs where the forcing claim cannot be initiated by a singleton triple.
    Since $F$ is a partial Steiner triple system, every edge of $F$ covers three distinct vertex-pairs and no pair is covered twice, so
    \begin{equation}\label{eq:leave-upper-new}
        3s+\ell=\binom{n}{2}.
    \end{equation}

    We now turn to the lower bound. We use the structure of $F$ to analyze the non-unique colors according to how they interact with the pairs covered by \(F\).  This will eventually produce a lower bound for \(I+\rho\).

    We partition the non-unique colors into good colors and bad colors.
    A non-unique color $\alpha$ is called {\it good} if every $\alpha$-colored triple
    has all its three pairs covered by $F$.  Otherwise, $\alpha$ is called {\it bad}.

    We first show that every good color has support size at least seven.
    Let $\alpha$ be a good color and take an $\alpha$-colored triple $abc$.
    Since $\alpha$ is good, there exist vertices $d,e,f$ such that $abd,ace,bcf$ are edges of $F$.
    By \cref{clm:forcing-new}, applied to the triples
    $abc$ and $abd,ace,bcf$, respectively, the six triples
    \[
        acd,\ bcd,\ abe,\ bce,\ abf,\ acf
    \]
    all have color $\alpha$.
    Since $\alpha$ is good and $acd$ is $\alpha$-colored, the pair $cd$ is
    covered by $F$; say
    \[
        cdg\in E(F).
    \]
    Applying \cref{clm:forcing-new} to the $\alpha$-colored triple $acd$ and
    the singleton triple $cdg$ shows that $g\in V_\alpha$.
    The vertices $a,b,c,d,e,f,g$ are distinct.  Indeed, the vertices
    $d,e,f$ lie outside $\{a,b,c\}$, since $abc$ has the non-unique color
    $\alpha$ and hence is not in $F$.  The vertices $d,e,f$ are pairwise
    distinct; otherwise two of the singleton triples $abd,ace,bcf$ would
    share a pair.  Finally, $g\notin\{a,b\}$ because $acd$ and $bcd$ have
    color $\alpha$ and hence are not in $F$, while $g\notin\{e,f\}$ because
    $cdg$ would then share a pair with $ace$ or $bcf$; also
    $g\notin\{c,d\}$ because $cdg$ is a triple.  Therefore
    \begin{equation}\label{eq:good-support-new}
        |V_\alpha|\ge 7
    \end{equation}
    for every good color $\alpha$.

    We now deal with the bad colors.  Let $\cB$ be the set of bad colors.
    Recall that a non-unique color $\alpha$ is bad precisely when there exists
    an $\alpha$-colored triple containing at least one edge of the leave graph
    $L(F)$.  For every $\alpha\in\cB$, fix once and for all one such
    $\alpha$-colored triple, and denote it by $T_\alpha$.
    Thus
    \[
        c_G(T_\alpha)=\alpha
        \qquad\text{and}\qquad
        \binom{T_\alpha}{2}\cap E(L(F))\ne\emptyset.
    \]
    We call $T_\alpha$ the fixed {\it witness} triple of the bad color $\alpha$.
    Define
    \[
        h_\alpha=
        \left|\binom{T_\alpha}{2}\cap E(L(F))\right|.
    \]
    Since $T_\alpha$ is a triple containing at least one leave-edge, we have $1\le h_\alpha\le3$.

    For every leave-edge $e\in E(L(F))$, define
    \[
        K_e=\{\alpha\in\cB: e\subseteq T_\alpha\},
        \qquad
        k_e=|K_e|.
    \]
    Thus $K_e$ consists of the bad colors whose fixed witness triple contains
    the leave-edge $e$.  The families $K_e$ are not meant to form a partition
    of $\cB$: a bad color may belong to more than one of them.

    In the next claims, a unit counted by \(I\) means an ordered pair
    \((\lambda,x)\) with \(x\in V_\lambda\), and a unit counted by \(\rho\)
    means an ordered pair \((z,\lambda)\) such that color \(\lambda\)
    contains a monochromatic triangle in the link graph \(L_z\).

    We first estimate the {\it base contribution} of a bad color, namely the
    contribution coming from its fixed witness triple and from the pairs of
    that triple that are covered by \(F\).

    \begin{claim}\label{clm:bad-base-new}
        Let $\alpha\in\cB$, and write its fixed witness triple as $T_\alpha=xyz$.
        The base contribution of
        $\alpha$ is at least
        \begin{equation}\label{eq:bad-base-new}
            3+2(3-h_\alpha)=9-2h_\alpha
        \end{equation}
        units.
    \end{claim}

    \begin{proof}
        The three vertices $x,y,z$ of $T_\alpha$ are all in $V_\alpha$, so
        they give the three $I$-units
        \[
            (\alpha,x),\qquad (\alpha,y),\qquad (\alpha,z).
        \]

        Among the three pairs of $T_\alpha$, exactly $h_\alpha$ are
        leave-edges.  Hence exactly $3-h_\alpha$ pairs of $T_\alpha$ are
        covered by $F$.

        Let $p$ be one covered pair of $T_\alpha=xyz$, say $p=\{x,y\}$, so that $z$ is the unique vertex of $T_\alpha$ not contained in $p$.
        Since $p$ is covered by $F$, there is a vertex $d_p$ such that $xyd_p\in E(F)$.
        The vertex $d_p$ is not in $T_\alpha$, because otherwise
        $xyd_p=T_\alpha$, which would mean that the non-unique color $\alpha$
        is used on a singleton triple.

        By \cref{clm:forcing-new}, applied to the $\alpha$-colored triple
        $xyz$ and the singleton triple $xyd_p$, we have
        \[
            c_G(xzd_p)=c_G(yzd_p)=\alpha.
        \]
        Therefore $d_p\in V_\alpha$, giving the additional $I$-unit $(\alpha,d_p)$.

        Moreover, in the link graph $L_z$, the three graph-edges $xy, xd_p,yd_p$ all have color $\alpha$, since they correspond to the triples
        $xyz,xzd_p,yzd_p$.
        Hence color $\alpha$ contains a monochromatic
        triangle in $L_z$, giving the $\rho$-unit $(z,\alpha)$.

        These units are distinct as $p$ ranges over the covered pairs of
        $T_\alpha$.  For the $I$-units, if two distinct covered pairs
        $p,q\subseteq T_\alpha$ had $d_p=d_q$, then the two singleton triples $p\cup\{d_p\}$ and $q\cup\{d_q\}$ would share a pair, contradicting that $F$ is a partial Steiner triple
        system.  Also, each $d_p$ lies outside $T_\alpha$, so these units are
        distinct from $(\alpha,x),(\alpha,y),(\alpha,z)$.  For the
        $\rho$-units, distinct pairs of the triple $T_\alpha$ have distinct
        remaining vertices: if $T_\alpha=xyz$, then the pairs $xy,xz,yz$
        correspond respectively to the link vertices $z,y,x$.

        Hence each covered pair of $T_\alpha$ gives two new units, one in $I$
        and one in $\rho$.  Since there are $3-h_\alpha$ covered pairs, the
        base contribution is at least
        \[
            3+2(3-h_\alpha)=9-2h_\alpha.
        \]
    \end{proof}

    We now estimate the additional contribution coming from the leave-edges in the fixed witness triples.  Fix a leave-edge
    \[
        e=uv\in E(L(F)).
    \]
    For every $\alpha\in K_e$, recall that $e\subseteq T_\alpha$.  Since
    $T_\alpha$ is a triple, there is a unique vertex $x_{\alpha,e}$ such that
    \[
        T_\alpha=e\cup\{x_{\alpha,e}\}=uvx_{\alpha,e}.
    \]
    If $\alpha,\beta\in K_e$ are distinct, then
    $x_{\alpha,e}\ne x_{\beta,e}$; otherwise the same triple would have the
    two distinct colors $\alpha$ and $\beta$.

    Consider two distinct colors $\alpha,\beta\in K_e$.  On the $4$-set $\{u,v,x_{\alpha,e},x_{\beta,e}\}$, the triples $uvx_{\alpha,e}$ and $uvx_{\beta,e}$ have colors $\alpha$ and
    $\beta$, respectively.
    Hence each of the two remaining triples $ux_{\alpha,e}x_{\beta,e}$, $vx_{\alpha,e}x_{\beta,e}$ has color in $\{\alpha,\beta\}$; otherwise, together with
    $uvx_{\alpha,e}$ and $uvx_{\beta,e}$, it would form a rainbow copy of
    $F_4$.

    For $t\in\{u,v\}$ and distinct $\alpha,\beta\in K_e$, say that
    $\alpha$ wins over $\beta$ at $t$ with respect to $e$ if
    \[
        c_G(tx_{\alpha,e}x_{\beta,e})=\alpha.
    \]
    For every unordered pair $\{\alpha,\beta\}\subseteq K_e$ and every
    $t\in\{u,v\}$, exactly one of $\alpha$ and $\beta$ wins over the other at
    $t$ with respect to $e$.

    For $\alpha\in K_e$ and $t\in\{u,v\}$, define
    \[
        W_t^e(\alpha)
        =
        \{\beta\in K_e\setminus\{\alpha\}:
        \alpha \text{ wins over } \beta \text{ at } t \text{ with respect to } e\}.
    \]
    For every leave-edge $e=uv$, define
    \[
        \mathcal Q_e^I
        =
        \{(\alpha,x_{\beta,e}):
        \alpha\in K_e,\,
        \beta\in W_u^e(\alpha)\cup W_v^e(\alpha)\}
    \]
    and
    \[
        \mathcal Q_e^\rho
        =
        \{(x_{\alpha,e},\alpha):
        \alpha\in K_e,\,
        W_u^e(\alpha)\cap W_v^e(\alpha)\neq\emptyset\}.
    \]
    We first claim that every element of $\mathcal Q_e^I$ is counted by $I$, and every element of $\mathcal Q_e^\rho$ is counted by $\rho$.
    In fact, if $\beta\in W_u^e(\alpha)$, then $c_G(ux_{\alpha,e}x_{\beta,e})=\alpha$, so $x_{\beta,e}\in V_\alpha$ and $(\alpha,x_{\beta,e})$ is counted by $I$.
    The same argument applies if $\beta\in W_v^e(\alpha)$, using
    the triple $vx_{\alpha,e}x_{\beta,e}$.  Thus every element of
    $\mathcal Q_e^I$ is counted by $I$.
    Write $M_e$ for $|\mathcal Q_e^I|+|\mathcal Q_e^\rho|$.
    Now we count $M_e$ in another way.
    \begin{claim}\label{clm:bad-winner-contribution-new}
        We have
        \begin{equation}\label{eq:Me-def}
            M_e=
            \sum_{\alpha\in K_e}
            \left(
            |W_u^e(\alpha)\cup W_v^e(\alpha)|
            +
            \mathbf 1_{W_u^e(\alpha)\cap W_v^e(\alpha)\neq\emptyset}
            \right),
        \end{equation}
        and
        \begin{equation}\label{eq:Me-lower}
            M_e\ge
            \begin{cases}
                0,                    & k_e=0,   \\
                \binom{k_e}{2}+k_e-1, & k_e\ge1.
            \end{cases}
        \end{equation}
    \end{claim}
    \begin{proof}
        For fixed $\alpha$, the vertices $x_{\beta,e}$ with
        $\beta\in W_u^e(\alpha)\cup W_v^e(\alpha)$ are distinct, since the
        triples $uvx_{\beta,e}$ have distinct colors.  For different values
        of $\alpha$, the first coordinate of the ordered pair
        $(\alpha,x_{\beta,e})$ is different.  Hence
        \[
            |\mathcal Q_e^I|
            =
            \sum_{\alpha\in K_e}|W_u^e(\alpha)\cup W_v^e(\alpha)|.
        \]

        Now suppose that $W_u^e(\alpha)\cap W_v^e(\alpha)\neq\emptyset$.
        Choose $\beta\in W_u^e(\alpha)\cap W_v^e(\alpha)$.  Then
        \[
            c_G(ux_{\alpha,e}x_{\beta,e})
            =
            c_G(vx_{\alpha,e}x_{\beta,e})
            =
            \alpha.
        \]
        Together with $c_G(uvx_{\alpha,e})=\alpha$,
        this shows that in the link graph $L_{x_{\alpha,e}}$ the three edges $uv, ux_{\beta,e}, vx_{\beta,e}$ form a monochromatic triangle of color $\alpha$.
        Thus $(x_{\alpha,e},\alpha)$ is counted by $\rho$.
        This proves that every element of $\mathcal Q_e^\rho$ is counted by $\rho$.
        The elements of $\mathcal Q_e^\rho$ are distinct because their color coordinates
        $\alpha$ are distinct.  Therefore \eqref{eq:Me-def} holds.

        It remains to prove \eqref{eq:Me-lower}.  If $k_e=0$, then
        $M_e=0$.
        Assume $k_e\ge1$. At the endpoint $u$, every unordered pair
        $\{\alpha,\beta\}\subseteq K_e$ has exactly one winner.\footnote{Readers may consider that the ``win'' relation defines a tournament.}  Hence
        \[
            \sum_{\alpha\in K_e}|W_u^e(\alpha)|
            =
            \binom{k_e}{2}.
        \]
        This is the baseline contribution.

        At the endpoint $v$, at most one color $\alpha\in K_e$ can satisfy
        $W_v^e(\alpha)=\emptyset$.  Indeed, if two distinct colors
        $\alpha,\beta\in K_e$ both had empty $W_v^e$-sets, then one of them
        would have to win over the other at $v$, contradicting the emptiness
        of its $W_v^e$-set.  Therefore at least $k_e-1$ colors
        $\alpha\in K_e$ satisfy $W_v^e(\alpha)\neq\emptyset$.
        Fix such an $\alpha$.  If
        \[
            W_u^e(\alpha)\cap W_v^e(\alpha)\neq\emptyset,
        \]
        then the indicator term contributes one additional unit beyond
        $|W_u^e(\alpha)|$.  If
        \[
            W_u^e(\alpha)\cap W_v^e(\alpha)=\emptyset,
        \]
        then
        \[
            |W_u^e(\alpha)\cup W_v^e(\alpha)|
            =
            |W_u^e(\alpha)|+|W_v^e(\alpha)|
            \ge |W_u^e(\alpha)|+1.
        \]
        Thus every $\alpha$ with $W_v^e(\alpha)\neq\emptyset$ contributes at
        least one unit beyond the baseline $|W_u^e(\alpha)|$.

        Since this happens for at least $k_e-1$ colors, we obtain
        \[
            M_e
            \ge
            \binom{k_e}{2}+k_e-1.
        \]
    \end{proof}

    The next claim checks that the base units and the leave-edge units are disjoint.

    \begin{claim}\label{clm:bad-disjoint-new}
        The units counted in the base contribution
        \eqref{eq:bad-base-new} and the units in the families
        $\mathcal Q_e^I,\mathcal Q_e^\rho$, over all leave-edges
        $e\in E(L(F))$, may be counted together in $I+\rho$ without
        repetition.
    \end{claim}

    \begin{proof}
        We check repetitions within $I$ and within $\rho$ separately.  Units
        with different color coordinate are automatically distinct.  Hence it
        is enough to fix one bad color $\alpha$ and prove that no unit
        assigned to this fixed $\alpha$ is counted twice.

        Write $T_\alpha=abc$.
        First consider the units counted by $I$.  The base $I$-units are
        \[
            (\alpha,a),\qquad (\alpha,b),\qquad (\alpha,c),
        \]
        together with the units $(\alpha,d_p)$ coming from covered pairs
        $p\subseteq T_\alpha$, where
        \[
            p\cup\{d_p\}\in E(F).
        \]
        For every covered pair $p$, the vertex $d_p$ lies outside
        $T_\alpha$; otherwise $T_\alpha$ itself would be a singleton triple.
        If $p\ne q$ are two covered pairs of $T_\alpha$, then
        $d_p\ne d_q$, because otherwise the two singleton triples $p\cup\{d_p\}$ and $q\cup\{d_q\}$ would share a pair, contradicting that $F$ is a partial Steiner triple
        system.  Thus the base $I$-units assigned to $\alpha$ are distinct.

        Now consider an $I$-unit assigned to $\alpha$ from some leave-edge
        $e\subseteq T_\alpha$.  Write
        \[
            e=uv,
            \qquad
            T_\alpha=e\cup\{x_{\alpha,e}\}.
        \]
        Such a unit has the form $(\alpha,x_{\beta,e})$, where $\beta\in K_e\setminus\{\alpha\}$ and $T_\beta=e\cup\{x_{\beta,e}\}$.

        First, $x_{\beta,e}\notin T_\alpha$.  Indeed, if
        $x_{\beta,e}\in T_\alpha$, then, since $e\subseteq T_\alpha$, we
        would have $T_\beta=T_\alpha$,
        forcing the same triple to have two distinct colors $\beta$ and
        $\alpha$.

        Second, $x_{\beta,e}$ cannot equal any vertex $d_p$ coming from a
        covered pair $p\subseteq T_\alpha$.  After relabelling, suppose
        \[
            T_\alpha=abc,\qquad e=ab,\qquad T_\beta=aby,
        \]
        and suppose that a covered pair of $T_\alpha$, say $ac$, satisfies $acy\in E(F)$.
        Then the three triples $abc,aby,acy$ lie on the same $4$-set and form a copy of $F_4$.
        Their colors are pairwise distinct: $abc$ has color $\alpha$, $aby$ has color $\beta\ne\alpha$, and $acy$ has a singleton color.
        This gives a rainbow copy of $F_4$, impossible.
        The case where the covered pair is $bc$ is identical by relabelling.

        Third, two $I$-units assigned to $\alpha$ from the same leave-edge
        cannot have the same second coordinate.  If
        $x_{\beta,e}=x_{\gamma,e}$ for distinct $\beta,\gamma\in K_e$, then
        \[
            T_\beta=e\cup\{x_{\beta,e}\}
            =
            e\cup\{x_{\gamma,e}\}
            =
            T_\gamma,
        \]
        so the same triple would have two distinct colors.

        Fourth, two $I$-units assigned to $\alpha$ from two distinct
        leave-edges of $T_\alpha$ cannot have the same second coordinate.
        After relabelling, suppose the two leave-edges are $ab$ and $ac$, and
        suppose both produce the same vertex $y$.  Then there exist bad colors
        $\beta,\gamma$ such that
        \[
            T_\beta=aby,
            \qquad
            T_\gamma=acy.
        \]
        The colors $\alpha,\beta,\gamma$ are pairwise distinct: neither
        $\beta$ nor $\gamma$ equals $\alpha$, and $\beta=\gamma$ would force
        the fixed witness triple $T_\beta$ to be both $aby$ and $acy$.
        Hence, $abc,aby,acy$ form a rainbow copy of $F_4$, impossible.

        Therefore all $I$-units assigned to the fixed color $\alpha$ are
        distinct.

        It remains to check the units counted by $\rho$.  For a pair
        $p\subseteq T_\alpha$, let $z_p$ denote the unique vertex of $T_\alpha\setminus p$.
        Thus, if $T_\alpha=abc$, then
        \[
            z_{ab}=c,\qquad z_{ac}=b,\qquad z_{bc}=a.
        \]

        A covered pair $p\subseteq T_\alpha$ contributes the $\rho$-unit $(z_p,\alpha)$,
        because the monochromatic triangle of color $\alpha$ lies in the link graph $L_{z_p}$.
        A contribution from a leave-edge
        $e=p\subseteq T_\alpha$ through the family $\mathcal Q_e^\rho$ also
        has the form $(z_p,\alpha)$,
        since for $p=uv$ and $T_\alpha=uvz_p$, the corresponding
        $\rho$-unit is $(z_p,\alpha)$.

        However, the same pair $p\subseteq T_\alpha$ cannot be both covered by
        $F$ and a leave-edge of $F$.  Hence the base contribution and the
        leave-edge contribution cannot assign the same $\rho$-unit for the
        same pair $p$.  Moreover, distinct pairs of the triple $T_\alpha$ have
        distinct vertices $z_p$.  Therefore all $\rho$-units assigned to
        $\alpha$ are distinct.

        Since the choice of $\alpha$ was arbitrary, and units with different
        color coordinate are distinct, all the counted units may be counted
        together in $I+\rho$ without repetition.
    \end{proof}

    Let $g$ be the number of good colors.  From the contributions of singleton
    colors, good colors, and bad colors, we obtain
    \begin{equation}\label{eq:I-rho-lower-new}
        I+\rho
        \ge
        3s+7g
        +
        \sum_{\alpha\in\cB}(9-2h_\alpha)
        +
        \sum_{e\in E(L(F))}
        \left(\binom{k_e}{2}+k_e-1\right)_+,
    \end{equation}
    where $x_+=\max\{x,0\}$.

    Let $b=|\cB|$ be the number of bad colors.  Then $c(G)=s+g+b$.
    By \eqref{eq:I-rho-lower-new},
    \begin{align}
        7c(G)
         & =
        7s+7g+7b \notag \\
         & \le
        (I+\rho)+4s
        +
        \sum_{\alpha\in\cB}(2h_\alpha-2)
        -
        \sum_{e\in E(L(F))}
        \left(\binom{k_e}{2}+k_e-1\right)_+.
        \label{eq:seven-c-new}
    \end{align}

    We next rewrite the two remaining sums in \eqref{eq:seven-c-new} as a
    single sum over the leave-edges of \(F\).  For a bad color \(\alpha\), the
    fixed witness triple \(T_\alpha\) contains exactly \(h_\alpha\)
    leave-edges.  We therefore distribute the term \(2h_\alpha-2\) equally
    over these \(h_\alpha\) leave-edge incidences.  Thus each incidence
    \(\alpha\in K_e\) receives weight
    \[
        \frac{2h_\alpha-2}{h_\alpha}
        =
        2-\frac2{h_\alpha}.
    \]
    Equivalently, since a fixed bad color \(\alpha\) belongs to \(K_e\) for
    exactly \(h_\alpha\) leave-edges \(e\subseteq T_\alpha\), a double count
    of these weighted incidences gives
    \begin{equation}\label{eq:weighted-incidence-new}
        \sum_{\alpha\in\cB}(2h_\alpha-2)
        =
        \sum_{e\in E(L(F))}
        \sum_{\alpha\in K_e}
        \left(2-\frac2{h_\alpha}\right).
    \end{equation}
    Since \(h_\alpha\in\{1,2,3\}\), every such weight is at most \(4/3\).

    For a fixed leave-edge \(e\), define
    \[
        R_e=
        \sum_{\alpha\in K_e}
        \left(2-\frac2{h_\alpha}\right)
        -
        \left(\binom{k_e}{2}+k_e-1\right)_+,
    \]
    where $x_+=\max\{x,0\}$.
    By \eqref{eq:weighted-incidence-new},
    \begin{equation}
        \sum_{\alpha\in\cB}(2h_\alpha-2)
        -
        \sum_{e\in E(L(F))}
        \left(\binom{k_e}{2}+k_e-1\right)_+
        =
        \sum_{e\in E(L(F))}R_e.
    \end{equation}

    Since $h_\alpha\in\{1,2,3\}$, this weight is at most $4/3$.
    We claim that
    \begin{equation}\label{eq:Re-bound-new}
        R_e\le \frac{4}{3}.
    \end{equation}
    If $k_e=0$, then $R_e=0$.  If $k_e=1$, then $R_e\le4/3$.  If $k_e=2$,
    then
    \[
        R_e\le \frac{4}{3}+\frac{4}{3}-2=\frac23.
    \]
    If $k_e\ge3$, then
    \[
        R_e
        \le
        \frac{4}{3} k_e-\left(\binom{k_e}{2}+k_e-1\right)
        \le \frac{4}{3}.
    \]
    This proves \eqref{eq:Re-bound-new}.

    Combining \eqref{eq:seven-c-new} and \eqref{eq:Re-bound-new}, we get
    \[
        7c(G)\le (I+\rho)+4s+\frac{4}{3}\ell.
    \]
    By \eqref{eq:I-rho-upper} and \eqref{eq:leave-upper-new},
    \[
        7c(G)
        \le
        n(n-2)+4s+\frac{4}{3}\ell
        =
        n(n-2)+\frac{4}{3}\binom n2.
    \]
    Therefore
    \[
        c(G)
        \le
        \frac17\left(n(n-2)+\frac{4}{3}\binom n2\right)
        =
        \frac{5n^2-8n}{21}.
    \]
    Since $G$ was arbitrary, this proves the theorem.
\end{proof}

\section{Concluding and further problems}\label{sec:concluding}
In this paper, we studied anti-Ramsey problems for cancellative configurations in complete $p$-graphs.
Let $\cF^{(p)}$ be the family of $p$-graphs consisting of three edges $A,B,C$ satisfying $A\triangle B\subseteq C$.
For every $p\ge3$ and $n\ge p+1$, we determined $\ar(n,\cF^{(p)})$ and characterized all extremal colorings.
In particular, \cref{thm:rainbow cancellative} shows that every rainbow $\cF^{(p)}$-free coloring of $K_n^{(p)}$ uses at most $1+\floor{n/p}$ colors, and equality holds if and only if after removing $n-p\cdot\floor{n/p}$ vertices,
the coloring contains $\floor{n/p}$ vertex-disjoint edges of pairwise distinct colors together with one additional `background' color on all remaining edges.

For $p=3$, rainbow cancellative is equivalent to forbidding rainbow copies of $F_4$ and $F_5$.
Our second main result, \cref{thm: F5 free}, shows that for $n\ge5$ the rainbow $F_5$-freeness already implies rainbow cancellative, and hence $\ar(n,F_5)=1+\floor{n/3}$.

Motivated by \cref{thm: F5 free}, it is natural to ask whether this theorem can be generalized to higher uniformities.
Let $\cO^{(p)}$ be the subfamily of $\cF^{(p)}$ consisting of three edges $A,B,C$ with $A\triangle B\subseteq C$ and $C\setminus (A\cup B)\neq\emptyset$; note that $\cO^{(3)}=\{F_5\}$.

\begin{problem}\label{pb1}
Is it true that every rainbow $\cO^{(p)}$-free edge-coloring of $K_n^{(p)}$ is rainbow $\cF^{(p)}$-free for all integers $p\ge3$ and $n\ge2p-1$?
\end{problem}

For rainbow $F_4$-free colorings of $K_n^{(3)}$, our constructions give
\[
    \ar(n,F_4)\ge m(n)+1
\]
for all $n\ge4$ via maximum partial Steiner triple systems, and for $n=2^s-1$,
\[
    \ar(n,F_4)\ge m(n)+\frac{n^2}{42}+o(n^2).
\]
Since $m(2^s-1)=(2^s-1)^2/6+O(2^s)$, this yields
\[
    \ar(2^s-1,F_4)\ge \frac{4}{21}(2^s-1)^2+o\bigl((2^s-1)^2\bigr).
\]
On the upper-bound side, we proved that every rainbow \(F_4\)-free
edge-coloring of \(K_n^{(3)}\) uses at most
\[
    \frac{5n^2-8n}{21}
\]
colors. The proof combines a Gallai-type estimate for the link graphs
with a charging argument based on singleton-colored triples and the
distinction between good and bad colors. This gives a quadratic upper
bound with leading coefficient \(5/21\).
Thus determining the correct quadratic coefficient of \(\ar(n,F_4)\),
even along the subsequence \(n=2^s-1\), remains open.
\begin{problem}\label{pb:F4-projective-subsequence}
Let \(n_s=2^s-1\). Does the limit
\[
    \lim_{s\to\infty}\frac{\ar(n_s,F_4)}{n_s^2}
\]
exist? If so, determine its value.
\end{problem}

\end{document}